\documentclass[11pt, reqno]{amsart}
\usepackage[colorlinks=true,pagebackref,hyperindex]{hyperref}
\usepackage{mathrsfs} 
\usepackage[all]{xy}
\usepackage{color}
\usepackage{amsfonts}
\usepackage{amscd}
\definecolor{darkgreen}{rgb}{0.0, 0.7, 0.0}

\definecolor{cyan}{cmyk}{1,0,0,0}

\newcommand{\bdg}{\begin{dg}}
\newcommand{\edg}{\end{dg}}

\newtheorem{tm}{Theorem}[subsection]
\newtheorem{lm}[tm]{Lemma}
\newtheorem{pr}[tm]{Proposition}
\newtheorem{rmk}[tm]{Remark}
\newtheorem{cor}[tm]{Corollary}
\newtheorem{ex}[tm]{Example}

\newtheorem{fact}[tm]{Fact}
\newtheorem{??}[tm]{Question}
\newtheorem{defi}[tm]{Definition}

\newcommand\blacksquare{{\hspace*{\fill} $\fbox{}$}} 
\newcommand{\ben}{\begin{enumerate}}
\newcommand{\een}{\end{enumerate}}
\newcommand{\bit}{\begin{itemize}}
\newcommand{\eit}{\end{itemize}}
\newcommand{\beq}{\begin{equation}}
\newcommand{\eeq}{\end{equation}}
\newcommand{\la}{\label}
\newcommand{\n}{\noindent}
\newcommand\ci{\cite}

\font\tenmsb=msbm10
\font\sevenmsb=msbm7
\font\fivemsb=msbm5
\newfam\msbfam
\textfont\msbfam=\tenmsb
\scriptfont\msbfam=\sevenmsb
\scriptscriptfont\msbfam=\fivemsb
\def\Bbb#1{{\fam\msbfam #1}}
\font\teneufm=eufm10
\font\seveneufm=eufm7
\font\fiveeufm=eufm5
\newfam\eufmfam
\textfont\eufmfam=\teneufm
\scriptfont\eufmfam=\seveneufm
\scriptscriptfont\eufmfam=\fiveeufm
\def\frak#1{{\fam\eufmfam\relax#1}}

\newcommand{\coke}{ \hbox{\rm Coker} }
\newcommand{\im}{ \hbox{\rm Im} }
\newcommand{\ke}{ \hbox{\rm Ker} }
\newcommand{\lorw}{\longrightarrow}

\newcommand\oql{\overline{\Bbb Q}_\ell}

\newcommand\zed{{\Bbb Z}}

\newcommand\s{\sigma}

\newcommand\e{\epsilon}

\newcommand{\w}[1]{\widetilde{#1}}

\newcommand{\ov}[1]{\overline{#1}}

\newcommand{\m}[1]{\mathcal{#1}}
\newcommand{\ms}[1]{\mathscr{#1}}

\newcommand{{\socle}}{\rm Socle}
\newcommand{\cu}[2]{\m{#1}_{#2}}

\title{A support theorem for the Hitchin fibration:
\\ the case of $SL_n$}
\author{
Mark Andrea A.  de Cataldo}
\address{Mark Andrea A. de Cataldo, Department of Mathematics, Stony Brook University, Stony Brook, NY, 11794-3651}
\thanks{
The research of M.A. de Cataldo was partially supported by NSF grants DMS-1301761 and DMS-1600515,  and by a grant from the Simons Foundation ($\#$296737 to Mark Andrea de Cataldo).
}
\date{}

\begin{document}

\begin{abstract}
We prove that the direct image complex for the  $D$-twisted $SL_n$ Hitchin fibration
is determined by its restriction to the elliptic locus, where the spectral curves are integral. The analogous result
for $GL_n$ is due to  P.-H. Chaudouard
and G. Laumon.
Along the way, 
we prove that the Tate module  of the relative Prym group scheme is polarizable,
and we also prove 
$\delta$-regularity results for some  auxiliary weak abelian  fibrations.
\end{abstract}

\maketitle


\newpage

\tableofcontents

\section{Introduction}\la{intro}

Let $C$ be a nonsingular projective and integral curve  of genus $g$ over   an algebraically closed field 
 of characteristic zero. Let $D$ be a line bundle   on $C,$ with $d:= \deg{(D)} > 2g-2$.

Fix a pair of coprime positive integers $(n,e)$. The  $GL_n$ moduli space  we consider is  the   moduli space \ci{nitsure} of  stable, rank $n$,  degree $e$, $D$-twisted  Higgs bundles $(E, \phi: E \to E(D))$ on $C$;
%see $\S$\ref{sndf}. 
it is an integral, quasi projective and nonsingular variety.
There is   the projective  Hitchin morphism $h_n: M_n \to A_n=\oplus_{i=1}^nH^0(C,iD)$ onto the affine space of the possible
characteristic polynomials of $\phi$.

The decomposition theorem 
\ci{bbd} predicts that the direct image complex ${Rh_{n}}_* \oql$ splits into a  finite direct sum of shifted 
simple perverse sheaves, each supported on an integral closed subvariety $S\subseteq A_n$.
These subvarieties are called the   supports of ${Rh_{n}}_* \oql$.
 The {\em socle of  ${Rh_{n}}_* \oql$}, denoted by  ${{\socle}} ({Rh_{n}}_* \oql)$, is the finite subset of $A_n$  of  generic points $\eta_S$ of the supports $S$
 of  ${Rh_{n}}_* \oql$. 
 
 One of the main geometric ingredients of B.C. Ng\^o's proof \ci{ngofl} of the Langlands-Shelstad fundamental lemma
 for reductive Lie groups $G,$
 is his support theorem \ci[Thm. 7.2.1]{ngofl}. This  is a statement concerning  the socle of the
 direct image complex via the Hitchin morphism $M_G \to A_G$ associated with $(G,C,D)$, after restriction to a
 certain  large open subset of the target $A_G$.
In the special case $G= GL_n$,  one considers the elliptic locus, i.e. the  dense open subvariety $A_n^{\rm ell} \subseteq A_n$ corresponding to those points $a \in A_n$ for which the associated spectral curve  is geometrically  integral.
Then   Ng\^o support theorem  implies that   
${\socle} (Rh_{n,*} \oql)  \cap A_n^{\rm ell}= \{\eta_{A_n}\}$, the generic point of the target $A_n.$
In other words, over the elliptic locus,   the simple summands appearing in the decomposition theorem are the intermediate extensions to $A_n^{\rm ell}$  of the direct image lisse sheaves over the locus  $A_n^{\rm smooth}$ of regular values of $h_n.$
 This has striking consequences for the
handling of orbital integrals over the elliptic locus (for every $G$),  which thus  become more tractable:
the ones corresponding to points in $A^{\rm ell}_n \setminus A^{\rm smooth}$  can be related to the ones over $A^{\rm smooth}_n$ by a principle of continuity
on $A_n^{\rm ell};$ this is precisely because
there are no  new supports on the boundary $A_n^{\rm ell}\setminus
A_n^{\rm smooth}$  (cf. \ci[\S1]{ngoaf}). 

Support-type theorems have been appearing in the related geometric contexts of  relative Hilbert schemes
and of  relative compactified Jacobians
of families of reduced  planar curves in \ci{Ma-Yu,Mi-Sh,Mi-Sh2,Mi-Sh-Vi,Sh}, also in connection
with BPS invariants.

It is  thus interesting, important, and seemingly non-trivial, to ``go beyond the elliptic locus."
P.-H.  Chaudouard and G. Laumon have extended 
\ci{ch-laflp2} Ngo's result on $A_n^{\rm ell}$  (which holds for every $G$),
   by proving  that (and here we specialize their result to $G=GL_n$) ${\socle} (Rh_{n,*} \oql)  \cap A_n^{\rm grss}= \{\eta_{A_n}\}$, where   $A_n^{\rm ell} \subseteq A_n^{\rm grss}$ is the larger  open locus for which the  associated spectral curves are reduced.
They have also  subsequently extended this result  to the whole base $A_n$ of the $D$-twisted  $GL_n$ Hitchin fibration in   \ci{ch-la}, where they prove the following
 \begin{tm}\la{ch-la-mt}
{\rm ({\bf $GL_n$ socle} \ci{ch-la})}
${{\socle}} ({R{{h}_n}}_* \oql) = \{\eta_{A_n}\}.$
\end{tm}
 
In particular, there  are no  new supports
 as one passes from the regular locus $A_n^{\rm smooth}$, to  the elliptic locus $A_n^{\rm ell}$,
 to $A_n^{\rm grss}$ and, finally,  to the whole of $A_n$.
The decomposition theorem then  takes the form of an isomorphism
 ${Rh_{n}}_*  \oql \cong \oplus_{q \geq 0} \m{IC}_{A_n} (R^q) [-q]$, where $R^q$ is the lisse  restriction
 of the $\oql$-constructible sheaf ${Rh_{n}}_* \oql$ to  $A_n^{\rm smooth},$ and where  $\m{IC}$ denotes
 the intermediate extension functor shifted so as to ``start" in cohomological degree zero.
 Since the general fibers of $h_n$ are (connected)  abelian varieties, we even have $R^q
 \cong \bigwedge^q R^1$ for every $0 \leq q \leq 2  d_{h_n},$ where $d_{h_n}$ is the relative dimension of $h_n.$

When  $G= SL_n$, we have the following picture, which goes back, at least implicitly, to \ci{nitsure}; see \S\ref{sndf}.  
Our  $SL_n$ moduli space $\check{M}_n \subseteq M_n$ consists of those stable
pairs
with fixed $\e=\det (E)$ and trivial trace ${\rm tr}(\phi)=0$. Then 
$\check{M}_n $ is an  integral, quasi projective and nonsingular variety. 
The restriction of the Hitchin morphism  $h_n,$  yields  the Hitchin morphism $\check{h}_n: \check{M}_n \to \check{A}_n:= \oplus_{i=2}^n H^0(X,iD),$ whose socle is the object of study of this paper.

This socle is known over the ellitpic locus $\check{A}_n^{\rm ell} =\check{A}_n \cap A_n:$ 
by work of  Ng\^o \ci{ngoend, ngofl}, we have that    ${\rm socle} ( {R{{\check h}_n}}_* \oql) \cap 
\check{A}_n^{\rm ell}$  is  given by the generic point $\eta_{\check{A}_n},$ union a finite set of points
(\ref{mopz}),
directly related to the endoscopy theory of $SL_n.$

 The purpose of this paper is to prove the following theorem, to the effect that 
 there are no new supports in $\check{A}_n\setminus \check{A}_n^{\rm ell}$, beyond the ones
(\ref{mopz})
already known to dwell in $\check{A}_n^{\rm ell}.$
\begin{tm}\la{maintm}
{\rm ({\bf $SL_n$ socle})}
${{\socle}} ({R{{\check h}_n}}_* \oql) \subseteq \check{A}_n^{ell}.$
\end{tm}

At first sight, the proof of our main Theorem \ref{maintm} for the $SL_n$ socle runs in parallel with the one 
of Theorem \ref{ch-la-mt} for the $GL_n$ socle in \ci[\S 9]{ch-la}, where  the authors   use: Ng\^o support inequality
over the whole base  $A_n;$  a multi-variable $\delta$-regularity inequality for  the Jacobi group scheme acting on the Hitchin fibers over
the elliptic locus;  the identity  between the abelian variety parts of the Jacobian of an arbitrary
spectral curve, and the Jacobian of the  normalization of its reduction.

The situation over $SL_n$ presents some substantial differences, which we now summarize.
\ben
\item
We need to prove the support inequality Theorem \ref{f4}.(1) over the  whole $SL_n$ base  $\check{A}_n.$
This had been known \ci{ngofl} over $\check{A}_n^{\rm ell}$  only.

\item
In order to achieve the $SL_n$ support inequality, we need to establish the polarizability 
Theorem \ref{slnpol} of the Tate module of the Prym group scheme over $\check{A}_n.$

\item
In turn, this required that:   we determine the explicit form (\ref{q10}) of a natural  polarization of the  Tate module of the  Jacobian
of an arbitrary spectral curve (see the $GL_n$ polarizability
Theorem \ref{poltmd}); we combine the explicit (\ref{q10})  
with the identification (\ref{lk}) of the affine parts of the fibers of the  Jacobi  and Prym groups schemes.
At this juncture, the $SL_n$ polarizability result follows by first exhibiting the Prym Tate module
as a natural direct summand of the Jacobi Tate module, and then  by using that pull-back and push-forward (norm)
are adjoint for the cup product.

\item
The $\delta$-regularity inequality over $\check{A}_n^{\rm ell}$ afforded by (\ref{98bis}) is not useful towards proving our main result
Theorem \ref{maintm}.
However, the method of proof  is: we use a product formula for the Hitchin fibration, and the identification   (\ref{lk}) of the affine parts of the  Jacobi and Prym
varieties, to show that the codimensions of the $\delta$-loci  are preserved when passing from the
elliptic locus $A_n^{\rm ell},$ to  the traceless elliptic locus $\check{A}_n^{\rm ell},$ so that (\ref{98bis}) holds.

\item
We pursue the same line of argument to reach the correct  $SL_n$ replacement (\ref{speremb}) of the 
 $GL_n$  multi-variable  $\delta$-regularity inequality used in \ci[\S9]{ch-la}.
This is done by first considering a multi-variable Hitchin base,  then by slicing it using linear weighted conditions
on the traces, and finally by  verifying that the codimensions of the $\delta$-loci are un-effected by the slicing.

\item
We fix a minor inaccuracy in \ci{ch-la}. See Remark \ref{nt}.

\een

As to the structure of the paper, we refer the reader to the  summaries at the beginning of
each of the five sections.

{\bf Acknowledgments.} I thank  Pierre-Henri Chaudouard, Brian Conrad,  Jochen Heinloth, Andrea Maffei, Eyal Markman, Luca Migliorini, Mircea Musta\c{t}\u{a},
Christian Pauly, Giulia Sacc\`a and   Jason Starr
for very stimulating conversations. I am very grateful to the anonymous referee for suggestions on improving the article.
\section{Preliminaries}\la{prelim}

This \S\ref{prelim} is  a collection of preliminary constructions, results and definitions.
\S\ref{aninnov}, \S\ref{sndf} 
introduce the $D$-twisted $SL_n$ Hitchin morphism  $\check{h}_n: \check{M}_n\to \check{A}_n$
which is the focus of this paper. The $GL_n$ case plays an important role, and is thus  discussed as well.  \S\ref{spcv} discusses spectral curves and covers: diagram
(\ref{gh}) plays a recurrent  role in the paper.  Spectral curves afford an important alternative
interpretation of the fibers of the Hitchin morphism via the Hitchin, Beauville-Narasimhan-Ramanan, Schaub correspondence, which is discussed in \S\ref{mmee}, together with some essential properties of the Hitchin morphism
and of its fibers: connectivity,  action of the Prym variety (\ref{rb}), irreducible components over the elliptic locus.
This leads to a discussion  in \S\ref{vbn} of the endoscopic locus  for $SL_n,$  which can be described with the aid
of the $n$-torsion in ${\rm Pic}^0(C)$. Section \S\ref{w2} discusses Ng\^o's notion of  $\delta$-regular weak abelian fibration, which is
a very important tool in the study of Hitchin systems, and an  essential one for this paper; two highlights are Ng\^o support inequality, and its ``opposite", the   $\delta$-regularity inequality.

Unless otherwise mentioned, we work with   varieties --separated schemes of finite type-- over a field of characteristic zero.
Let $C$ be an integral and  nonsingular curve of genus $g$ and let   $D \in {\rm Pic}^d(C)$ be a  fixed line bundle on $C$ of degree
$d > 2g-2.$ We fix  two coprime integers $(n,e)$ and a  degree $e$ line bundle $\e \in {\rm Pic}^e(C)$. Recall that
the co-primality condition ensures that the two notions of  stability and of semi-stability coincide,
so that the (coarse=fine) moduli spaces of Higgs bundles  we consider are nonsingular.

\subsection{$GL_n$ and $SL_n$ Hitchin fibrations}\la{aninnov} $\;$

A standard reference for what follows is \ci{nitsure}. 

{\bf The $GL_n$ case.}
Let $\ms{M}$ be the moduli space of stable, $D$-twisted, $GL_n$ Higgs bundles of rank $n$ and degree $e$ on the curve $C$.
Then  $\ms{M}$ is a nonsingular and  quasi-projective variety
of pure dimension $n^2d+1$. It parameterizes stable  pairs
$(E,\phi)$, where:  $E$ is a rank $n$ and degree $e$ vector bundle on the curve $C$, and  $\phi: E \to E(D)$ is a morphism
of $\mathcal{O}_C$-modules.
The notion of stability is the usual one:  for every $\phi$-invariant proper sub-bundle $F\subseteq E$,
the slopes $\mu:= \deg/{\rm rk}$ satisfy the inequality
 $\mu (F) < \mu (E)$. There is the projective characteristic morphism 
\[
h: \ms{M} \to \ms{A}:= \oplus_{i=1}^n H^0(C, iD),
\]
sending  $(E,\phi)$ to the  coefficients $(- {\rm tr} (\phi), + {\rm tr}(\wedge^2 \phi), \ldots, (-1)^n \det (\phi))$
of the characteristic polynomial of $\phi$. The elements of $\ms{A}$ are called  characteristics.

 The pure-dimensional nonsingular variety  $\ms{M}$ is connected, hence  irreducible. One way to see this, is to
couple the fact that the proper characteristic morphism is of pure relative dimension 
(\ci[Corollaire 8.2]{ch-la}) with the fact (Remark \ref{rmkre}) that
the general fiber, being the Jacobian of a nonsingular and connected spectral curve,  is connected.
I thank the anonymous referee for bringing this to my attention.
%%%%% se dim rel pura, allora mappa deve essere dominante da ogni componente connessa;
%%%%% dominante e propria, dunque surj.
%%%%% ma la fibra gen e` connessa.
%%%%% dunque una sola comp connex.

 The moduli space $N$
of rank $n$ and degree $e$ vector bundles on $C$ sits naturally in $\ms{M}$ (take $\phi:=0$).  It is well-known that $N$ is integral, nonsingular, projective and of dimension
$n^2(g-1)+1$. We have inclusions
$\ms{M} = \overline{T}  \supseteq T \supseteq N$, where $T$ is the total space of the vector bundle
of rank $n^2 [d-(g-1)]$
over $N$ with fiber at $E$ given by $H^0(C, {\rm End}(E)(D))$; see \ci[Prop. 7.1 and the formula above it]{nitsure}. 
% I used the les above prop 7.1 for $\phi$=0 to get ses with $h^1=\dim N$; then I used RR for $End(E)(D)$ to compare;
%at the end I even  get $h^1 (End(D))=0$
Then  $T$ is integral, nonsingular, of dimension
$n^2d+1$, and it is  a Zariski-dense  open subvariety
of $\ms{M}$; see \ci[p.297-8]{nitsure}.

{\bf The $GL_n$ traceless case.} We need the following simple  traceless variant of the $D$-twisted $GL_n$ moduli space: geometrically, it 
is the pre-image via the morphism $h: \ms{M} \to \ms{A}$  of the locus $\ms{A}(0) \subseteq \ms{A}$ of traceless characteristics.
Let $\ms{M}(0) \subseteq \ms{M}$ be the moduli space of   stable pairs $(E,\phi)$ as above, subject to the additional traceless constraint ${\rm tr}(\phi)=0$. By repeating the arguments in \ci{nitsure} concerning $\ms{M}$,
but with the  traceless constraint, we see that  $\ms{M}(0)$  is a nonsingular and  quasi-projective variety,
of pure dimension $nd^2+1 - h^0(D).$ Morevoer, we have a natural isomorphism $\ms{M}\cong
H^0(C,D) \times \ms{M}(0)$ (see \S\ref{prd}, (\ref{codgsq})), implying that the nonsingular $\ms{M}(0)$ is connected and  irreducible.

As above, we have inclusions $\ms{M}(0) = \overline{T(0)} \supseteq T(0) \supseteq N$, with the same properties listed above, except that we 
take traceless endomorphisms, and the rank of the corresponding vector bundle on $N$ equals $h^0 (C, {\rm End}^0(E)(D))= n^2[d-(g-1)] - h^0(D)$. We have  the projective  characteristic morphism 
\[h(0) : \ms{M}(0) \to \ms{A}(0):= \oplus_{i=2}^n H^0(C,iD).\]

{\bf The $SL_n$ case.} Finally, we introduce the moduli space to which this paper is devoted.
Fix a line bundle $\e \in {\rm Pic}^e(C)$  on $C$, of degree $e.$
Let $\ms{M}(0,\e) \subseteq \ms{M}(0) \subseteq  \ms{M}$ be the moduli space of  stable pairs $(E,\phi)$ as above, subject to  ${\rm tr}(\phi)=0$ and to $\det (E) =\e$.
 By repeating the arguments in \ci{nitsure}, but with  the traceless and fixed-determinant constraints, we see that the variety $\ms{M}(0,\e)$  is  nonsingular and  quasi-projective,
of pure dimension $n^2d+1 - h^0(D)- g.$  We have  the projective  characteristic map
\[h(0,\e) : \ms{M}(0,\e) \to \ms{A}(0):= \oplus_{i=2}^n H^0(C,iD).\]
%and its restriction
%$h(0,\e)_o: \ms{M}(0,\e)_o \to \ms{A}(0).$
Let $\ms{M}(0,\e)_o$ be the irreducible (also a connected) component containing the moduli space $N(\e)$ of stable rank $n$ and degree $e$ bundles on $C$ with fixed determinant $\e \in {\rm Pic}^e(C)$.
It is well-known that the variety $N(\e)$ is integral, nonsingular, projective,  and of dimension  $(n^2-1)(g-1)$.
As above, we have inclusions $\ms{M}(0,\e)_o = \overline{T(0,\e)} \supseteq T(0,\e) \supseteq N(\e)$, with the same properties listed above (again, we  take traceless endomorphisms).

Note that $\ms{M}(0,\e)=\ms{M}(0,\e)_o$ and that  the isomorphism class of $\ms{M}(0,\e)_o$ is independent of $\e \in  {\rm Pic}^e(C)$.
This can be seen as in the proof of the following simple

\begin{lm}\la{rt6} The variety $\ms{M}(0,\e)$ is connected, i.e. $\ms{M}(0,\e)=\ms{M}(0,\e)_o$.
The variety $\ms{M}(0,\e)$ is the fiber over $\e \in {\rm Pic}^e(C)$ of the determinant map
$\det: \ms{M}(0) \to {\rm Pic}^e(C)$, as well as the fiber over $(0,\e) \in H^0(C,D)\times {\rm Pic}^e(C)$
of the trace-determinant map ${\rm tr} \times \det:\ms{M} \to H^0(C,D)\times {\rm Pic}^e(C)$. 
\end{lm}
\begin{proof}
The map $\det$ is    equivariant with respect
to the action of ${\rm Pic}^0(C)$ given by $L\cdot (E,\phi):= (E\otimes L, \phi \otimes {\rm Id}_L)$
on the domain, and by $L \cdot M:= M \otimes L^{\otimes n}$ on the target.  It follows that $\det$ is smooth
of relative dimension $\dim{(\ms{M}(0,\e))}$, and that all of its fibers are mutually isomorphic to each other. The same is true of the restriction of $\det$ to the ${\rm Pic}^0(C)$-invariant
open subvariety $T(0) \subseteq \ms{M}(0).$ Let $Z:= \ms{M} (0) \setminus T(0)$ be the closed complement. 
The resulting map $Z \to {\rm Pic}^0(C)$  is also  ${\rm Pic}^0(C)$-invariant, so that all of its fibers have the same dimension, which must be strictly smaller than  $\dim{(\ms{M}(0,\e))}$.
It is clear that $\ms{M}(0,\e)_o$ is contained in $\det^{-1} (\e)=\ms{M}(0\,e)$ and that, by the smoothness of $\det$, it must constitute a connected component
of such fiber. Since the fiber $\det^{-1} (\e)$ is of pure dimension  $\dim{(\ms{M}(0,\e))}$, the variety $Z$ cannot contain
any other connected component of  the smooth fiber $\det^{-1} (\e)$. We have thus proved 
that $\det^{-1}(\e) = \ms{M}(0,\e)=\ms{M}(0,\e)_o$, which are thus all connected, for the third one is by construction.
The assertion concerning ${\rm tr}\times \det$ is proved in a similar way.
\end{proof}

\subsection{Simplified notation for  Hitchin fibrations}\la{sndf}$\;$

 We want to simplify our notation, while emphasizing the role of the rank $n$.   
 
Fix $(n,e,\e, D)$. Denote the characteristic Hitchin   morphisms
\[h: \ms{M} \to \ms{A},  \qquad h(0): \ms{M}(0) \to 
\ms{A}(0),   \qquad
h(0,e) : \ms{M}(0,\e) \to \ms{A}(0)
\] 
as follows:
\beq\la{hitm}
h_n: M_n \to A_n, \qquad  h_n(0) : M_n(0) \to A_n(0), \qquad \check h_n: \check M_n \to \check A_n:=A_n(0).
\eeq
We are denoting the same object $\check{A}_n = A_n(0)$ in two different ways:
we prefer to use the notation $A_n(0)$ when dealing with $M_n(0)$,
and to use $\check{A}_n$ when dealing with $\check{M}_n$.

The projective morphisms   $h_n$ and $\check{h}_n$ are known  as the $D$-twisted, Hitchin $GL_n$ and $SL_n$ fibrations.  The morphism $h_n (0)$ plays an important auxiliary role in this paper.

We shall also need to consider two  several-variable-variants
of these Hitchin fibrations, namely  
$h_{n_\bullet}: M_{n_\bullet} \to A_{\bullet}$,  and $h_{n_\bullet m_\bullet}(0): M_{n_\bullet m_\bullet} (0)
\to A_{n_\bullet m_\bullet} (0)$ (cf. \S\ref{qwoz} and \S\ref{qw2}).

An important locus inside the base of the Hitchin fibration is the elliptic locus. In the case of $GL_n$ and $SL_n$ we define 
it as follows.

\begin{defi}\la{defell} {\rm ({\bf  Elliptic locus})}
The elliptic loci $A_n^{\rm ell} \subseteq A_n$  and $\check{A}_n^{\rm ell} \subseteq \check{A}_n$ are the respective  Zariski dense open subvarieties
of points such that the  associated spectral curves are geometrically integral.
\end{defi}

Clearly, $\check{A}_n^{\rm ell} = A_n^{\rm ell} \cap \check{A}_n$.

 \subsection{Spectral covers and the norm map}\la{spcv}$\;$

Let $\pi:V(D)\to C$ be the surface  total space of the line bundle $D$ on $C$.
Let $t$ be the universal section of $\pi^* D$,
with zero set on $V(D)$ given by $C$, viewed as the zero section on $V(D)$. Let $\mathcal C=\mathcal C_n \subseteq V(D) \times A_n$ be the universal spectral curve, that is the relative curve over $A_n$
with fiber $\mathcal C_a$ over  a closed point $a=(a(1), \ldots, a(n)) \in A_n,$ given by the zero set in $V(D)\times \{a\}$ of the section
$P_a(t):= t^n + \pi^*a(1)t^{n-1} + \pi^*a(2) t^{n-2} + \ldots +  \pi^*a(n)$ of  the line bundle $ \pi^*(nD)$ on  $V(D)\times \{a\}$.
Note that $A_n$ is an affine space inside the projective space  given by the linear system $|nC|$ on the standard projective completion
$\mathbb{P}_C (\m{O}_C \oplus \m{O}_C(-D))$
of $V(D)$, where $C$ sits as the zero section.
Let $p:\m{C} \to A_n$ be the natural ensuing morphism.
For $a \in A_n,$ the  spectral curve $\mathcal C_a$ is geometrically connected  and maps  $n:1$ onto $C_a:=C\otimes k(a)$
via the flat finite morphism $p_a:=p_{|\m{C}_a}: \mathcal C_a \to C_a$.
The total space of the family  $\mathcal C$ is integral and  nonsingular, and  the natural morphism
$\m{C} \to C\times A_n$ is finite, flat  and of degree $n$.

When we view each spectral curve $\mathcal C_a$  over a geometric  point $a $ of $A_n,$   as an effective Cartier divisor on $V(D)\otimes k(a),$ we may write
$\mathcal C_a = \sum_{k=1}^s m_{k,a} \mathcal C_{k,a}$,
where each $\mathcal C_{k,a}$ is geometrically  integral, each  integer $m_{k,a} >0$, and the expression is unique.
Each curve $\mathcal C_{k,a}$ maps  finitely  onto $C_a;$ denote the corresponding degree by $n_{k,a}.$ Clearly, $n= \sum_k m_{k,a} n_{k,a}.$
By considering the coefficients $a(i)$ above as the  $i$-th symmetric functions of the $D$-valued roots
of the polynomial equation $P_a(t)$, we obtain the unique factorization $P_a(t)= \prod_{k} P^{m_{k,a}}_{a_k}(t)$, where each $a_k(i) \in H^0(C,iD)$, $1\leq i\leq n_k,$ is the $i$-th symmetric function
of the $D$-valued roots of $P_a(t)$  that lie on $\mathcal C_{k,a}$. In particular,  we have that $a_k$ is a geometric point of $A_{n_k}$ (base of the Hitchin fibration for $(n_k, e, D)$), and that  $\m{C}_{k,a}$ is a spectral curve for the $D$-twisted $GL_{n_k}$ Hitchin fibration. 
%In fact, since by construction $\mathcal C_{k,a}$ is geometrically integral, we have that $a_k$ is a geometric point of  $A_{n_k}^{\rm ell}$.  

Let $a \in A_n$.
 We need to list the various covers of the curve $C_a=C \otimes k(a)$ that arise from the given spectral cover $p_a: \mathcal{C}_a \to C_a.$ In doing so, we also simplify and abuse  the notation a little bit.
We do not assume the point $a \in A_n$ to be a geometric one,
so that the intervening integral curves may not be geometrically integral.

We denote the curve  $\m{C}_a= \sum_k m_k  \Gamma_k$, where: each   $\Gamma_k$
is a spectral curve,  zero-set of a section $\frak{s}_k$  of the line bundle $\pi^* (n_kD)$ on the surface $V(D)\otimes k(a)$;  the $n_k >0$ are uniquely-determined positive integers, and  we have $n = \sum n_k m_k.$ Scheme-theoretically, $m_k \Gamma_k$ is the zero set of the $m_k$-th
power $\frak{s}^{m_k}$, and $\m{C}_a = \sum_k m_k \Gamma_k$ 
is the zero set of the product $\prod_k \frak{s}_k^{m_k}$. We denote by $\xi_{3,k}: \w{\Gamma_k} \to \Gamma_k$ the normalization
morphism. 

We have the following commutative diagram of finite surjective  morphisms of curves
\beq\la{gh}
\xymatrix{
&&\sum_k \Gamma_k  = {\mathcal C}_{a,{\rm red}} \ar[rd]^\rho& \\
\w{\mathcal C_{a,{\rm red}}}=\coprod_k \w{\Gamma_k} \ar[rru]^\nu \ar@/^1pc/[rrr]^{\hskip 13mm \xi} \ar[r]_{\hskip
5mm \xi_3}  \ar[rrdd]^{\w{p}} & 
\coprod_k \Gamma_k \ar[rr]_{\xi_4} \ar[ru] \ar[rd]^{\xi_2} \ar@/_/[rdd]^{p''} && \sum_k m_k \Gamma_k = \mathcal C_a \ar[ddl]_p \\
&& \coprod_k m_k \Gamma_k \ar[ru]^{\xi_1} \ar[d]^{p'} & \\
&& C_a=C\otimes k(a).& &
}
\eeq

\begin{fact}\la{piczero}
{\rm ({\bf The Jacobian of a spectral curve})}
Let $\ov{a}$ be a geometric point of $A_n$ with underlying Zariski point $a\in A_n$. Then 
the identity connected component ${\rm Pic}^0 (\m{C}_a)$  of the degree zero component of ${\rm Pic}(\m{C}_a)$
consists of the isomorphism classes of line bundles on  the spectral curve $\m{C}_a$ whose restriction 
to each irreducible component of $\m{C}_{\ov{a}}$ have degree zero; see \ci[\S9.3, Cor. 13]{neron}.
\end{fact}

Each of the morphisms to $C_a$  in diagram (\ref{gh}) comes with an associated 
norm morphism  into ${\rm Pic}(C_a),$  and  with an  associated pull-back morphism from ${\rm Pic}(C_a).$ 
Similarly,
for ${\rm Pic}^0$'s. For the definition and properties of the norm morphism, see \ci[\S6.5]{EGAII} and \ci[\S21.5]{EGAIV.4}.
For a quick reference for the facts we use in this paper,  see also \ci[\S3]{ha-pa}.
See also Fact   \ref{norm}.
 We have    the norm morphism
 \beq\la{nme}
N_p : {\rm Pic} (\m{C}_a) \lorw {\rm Pic}(C), \;\;{\rm Pic}^0 (\m{C}_a) \lorw {\rm Pic}^0(C_a).
 \eeq
We  also have the norm morphisms $N_{\w{p}},
N_{p'}$ and $N_{p''},$ as well as the pull-back morphisms $\w{p}^*, {p'}^*$ and ${p''}^*;$
similarly, for  each of  their $k$-th component.

We end this section with the following  consideration that will play a role later. 

\begin{fact}\la{zeta}
Since $D$ has positive degree $d > 2g-2$ on $C$, we have that, on each $\Gamma_k,$ the line bundle 
 $({p''}^*n_k D)_{|\Gamma_k}$  admits  some nontrivial section $z_k$  with  zero  subscheme $\zeta_k$ supported at a closed  finite non-empty subset of $\Gamma_k$.  We fix such a section, and we obtain the short exact sequences of $\m{O}_{\Gamma_k}$-modules 
\beq\la{rv0}
0 \lorw \m{O}_{\Gamma_k}(-\Gamma_k)   \lorw  \m{O}_{\Gamma_k} \lorw \m{O}_{\zeta_k}  \lorw 0.
\eeq
\end{fact}

 \subsection{The fibers of the Hitchin fibrations}\la{mmee} $\;$

Let $a \in A_n$ and let $p: \m{C}_a=: \Gamma= \sum_k m_k\Gamma_k  \to C_a$ be the corresponding spectral cover,
with $n = n_\Gamma = \deg{(p)} = \sum_k n_k m_k  = \rm{rk}_C (p_* \m{O}_\Gamma)$;
see (\ref{gh}).  Let  $j_k: \eta_k \to \Gamma$
be the finitely many generic points in $\Gamma,$ one for each irreducible component $m_k \Gamma_k.$
A coherent sheaf $\m{E}$ on $\Gamma$ is  torsion free iff the natural map
$\m{E} \to \prod_k \m{E}_k$ is injective, where $\m{E}_k = {j_k}_* j_k^* \m{E};$ see \ci[Def. 1.1. and Prop. 1.1]{sc}.
A torsion free $\m{E}$ is said to have ${\rm Rk}_{\Gamma} (\m{E}) = r$ if
its lengths at the generic points satisfy $l_k (\m{E}):=l_{\m{O}_{\eta_k}} (\m{E}_{\eta_k})= r \,m_k,$
for every $k;$ such a rank is then a non-negative rational number, which is zero iff $\m{E}=0.$  A torsion free $\m{E}$ may fail to have a well-defined ${\rm Rk}_\Gamma (\m{E}).$
When  this rank is well-defined,     one defines
the  degree by setting  ${\rm Deg}_\Gamma (\m{E}) := \chi(\m{E})   - {\rm Rk}_\Gamma (\m{E})  \chi (\m{O}_\Gamma).$ 

Let $P_\Gamma =
\prod_k P_{\Gamma_k}^{m_k}$ be the characteristic equation defining $\Gamma.$
A torsion free  coherent sheaf $\m{E}$ on $\Gamma$ correspond, via $p_*,$ to a 
pair $(E,\phi:E\to E(D))$ on $C$, where: $E=p_* \m{E}$ is locally free of rank ${\rm rk}_C(E)=
\sum_k n_k l_k;$  $\phi$ is the twisted endomorphism corresponding to multiplication by $t$
on $\m{E}.$  Then $\phi$  has  characteristic polynomial $P_\phi=\prod_k P_{\Gamma_k}^{l_k}.$ 
It follows that $P_\phi = P_{\Gamma}$ iff  ${\rm Rk}_\Gamma (\m{E})$ is well-defined and equals $1$ (this is the content of
\ci[Prop. 2.1]{sc}). 

  Note that  \ci[\S3.3]{ha-pa} introduces, via Riemann-Roch, a different  notion of rank and degree for every coherent $\m{O}_\Gamma$-module,
even for those torsion-free ones for which the notion of degree given above is not well-defined.
In this paper, we use the notion of rank and degree  given above \ci{sc}, not the one in \ci{ha-pa}. The forthcoming modular description of the fibers of the Hitchin fibration is given in terms of the notions employed in this paper, and the torsion free sheaves on spectral curves that arise are, by necessity,
the ones for which the rank is well-defined and it has value one.

\begin{ex}\la{ew}
 Let $nC= \m{C}_0$ be the spectral curve for the characteristic polynomial $t^n$, i.e. for $a=0 \in A_n.$
 See {\rm \S\ref{spcv}} for the notation.

 For $1 \leq m \leq n$, we consider the curves $mC$, their structural sheaves  $\m{O}_{mC}$
 and their ideal sheaves $\m{I}_{mC,nC} \subseteq \m{O}_{nC}$.
 We have $\chi (\m{O}_{mC}) = -{m\choose 2}d - m(g-1);$ see {\rm (\ref{rr})}.
 We then have: ${\rm Rk}_{nC} (\m{O}_{mC}) = m/n$; ${\rm Rk}_{nC} (\m{I}_{mC, nC}) = 1-m/n$;
 ${\rm Deg}_{nC}(\m{O}_{mC})= \frac{m}{2} (n-m)d$; ${\rm Deg}_{nC}(\m{I}_{mC, nC})= -\frac{m}{2} (n-m)d.$
 We have $P(\m{O}_{mC})=P_C^m,$ $P(\m{I}_{mC,nC})=P_C^{m-n}.$

 Let $E$ be a  stable  vector bundle of rank $n$ and degree $e$
  on $C$; let $i:C\to nC$
 be the natural map induced by the zero section $C \to C\subseteq V$, followed by the closed embedding
 $C= (nC)_{\rm red} \to nC$; we have that ${\rm Rk}_{nC}(i_*E)=1$ and ${\rm Deg}_{nC}(i_*E) =  e + {n \choose 2}d$.  We have $P(i_* E)=P_{nC}=P_C^n.$

 \end{ex}

It is easy to show that in the context of torsion free and  ${\rm Rk}_\Gamma (-)=r$
coherent sheaves on $\Gamma,$  the notion of slope   in  \ci[p.55]{siI} and \ci[Cor. 6.9]{siII},
and the notion   of slope  ${\rm Deg}_\Gamma / {\rm Rk}_\Gamma$, 
yield coinciding  notions of slope stability.   In turn, this coincides with the notion of 
slope-stable   Higgs pair $(p_* \m{E},\phi),$
with slopes defined by taking  $\deg_C / {\rm rk}_C.$ By working with quotients, instead of with
subobjects, the stability condition takes the form (\ref{34}) below.  Define
 \beq\la{e'}
 e':= e + {n \choose 2}d.
 \eeq

 \begin{rmk}\la{compq}
 As pointed out in \ci[Rmk. 4.2]{ch-la}, the statement of \ci[Thm. 3.1]{sc}, which characterizes  stability, needs to be slightly modified {\rm (cf. {\rm (\ref{34})})}.
 \end{rmk}
 \begin{rmk}\la{compqbis}
 Let us point out that one has also to correct some minor  inaccuracies  at the  end of the proof of  \ci[Prop. 2.1, p. 303, from the  top,  to the  end of the proof]{sc}: the degrees on the finite maps from the reduced irreducible components of the spectral curve are omitted from the first two displayed equalities;   the inequality  on the lengths implying that the rank should be one is not justified. One  remedies this minor inaccuracies by means of the discussion 
 at the beginning of this section involving  the role of the characteristic polynomials.
 \end{rmk}

 {\bf Modular description of the Hitchin fiber $M_{n,a}:=h_n^{-1}(a)$, $a\in A_n$.}
  The discussion that follows does not require that one first  proves that $M_n$ is irreducible; in particular,
 it can used in order to establish this fact, as it has been done in \S\ref{aninnov}.
The Hitchin fiber $M_{n,a}:=h_n^{-1}(a)$, i.e. the moduli space 
% this requires a discussion, if requested. Alper paper 4.7.(i) and Ex 8.1 imply this. Caution: tame artin stacks are defined
%by [AbrOlVist:Tame stacks in positive char] in positive and mixed char by an exactness property that is automatic in %char zero; so we can say that in char zero we are tame and the formation of good muduli is compatible with base %change.
 of stable $D$-twisted Higgs pairs with  rank $n$ and degree $e$  and with characteristic $a\in A_n$, is isomorphic
 to the  moduli space of torsion free sheaves $\m{E}$ on the spectral curve $\m{C}_a$ with
 ${\rm Rk}_{\m{C}_a}(\m{E})=1$ (and hence  with associated characteristic polynomial $P_\phi=P_{\m{C}_a}$)
  and ${\rm Deg}_{\m{C}_a}(\m{E})= e',$ subject to the following stability condition:
 for every closed subscheme $i_Z: Z\to \m{C}_a$ of pure dimension one, 
 for every  torsion free quotient  $\m{O}_Z$-module $\xymatrix{i_Z^*E \ar@{->>}[r] &\m{E}_Z}$  with ${\rm Rk}_Z (\m{E}_Z)=1$, we have that 
 \beq\la{34}
\frac{ {\rm Deg}_Z (\m{E}_Z)}{{\rm Rk}_C (p_* \m{O}_Z)} +
\frac{1}{2} \left(n - {\rm Rk}_C (p_* \m{O}_Z)   \right) d > \frac{e'
%{\rm Deg}_{\m{C}_a} (\m{E})
}{n}.
 \eeq
 The isomorphism is given by the push-forward morphism ${p_a}_*$ on coherent sheaves under the finite, flat, degree $n$, spectral cover morphism 
 $p_a:\m{C}_a \to C_a=C\otimes k(a).$
 
 \begin{rmk}\la{rmkre}
 If the spectral curve $\m{C}_a$ is smooth, i.e. for $a \in A_n$ general, then the   fiber 
 $M_{n,a}$ is geometrically connected, for, in view of its modular description, it coincides with ${\rm Pic}^{e'}(\m{C}_a)$.
 \end{rmk}
 
 Let us record the  properties of the norm map that we need.
 \begin{fact}\la{norm}
Let $p_a: \m{C}_a \to C_a$ be  a spectral cover (of degree $n$) with norm map $N_{p_a}: {\rm Pic}^0(\m{C}_a)\to {\rm Pic}^0(C_a)$ and pull-back
map $p_a^*: {\rm Pic}^0(C_a)\to {\rm Pic}^0(\m{C}_a)$. For what follows, see \ci[Cor. 1.3 and \S3]{ha-pa}.

\ben
\item
For every  $L \in {\rm Pic}(C_a)$, we have  $N_{p_a} (p_a^* L) = L^{\otimes n};$ in particular, $N_{p_a}$ is surjective.
\item
Let $\m{E}$ be a torsion free $\m{O}_{\m{C}_a}$-module of some integral rank ${\rm Rk}_{\m{C}_a}(\m{E})=:r$
and let $\m{L} \in {\rm Pic} (\m{C}_a)$;
then $\det ({p_a}_*( \m{E} \otimes \m{L}) ) = \det ({p_a}_* \m{E}) \otimes N_{p_a}(\m{L})^{\otimes r};$
\item 
if  $a \in A_n$ is general, then   $\ke{(N_{p_a})}$ is a (connected) abelian variety
 (see {\rm \S\ref{vbn}}).
\een

\end{fact}

\begin{pr}\la{kj}
The projective, $D$-twisted, $GL_n$ Hitchin morphism  $h_n:M_n \to A_n$ is surjective, with geometrically connected fibers,
flat of pure relative dimension 
\beq\la{rr}
d_{h_n} =  {n\choose 2} d + n(g-1) +1.
\eeq 
Let $a \in A_n$. Then ${\rm Pic}^0(\m{C}_a)$ acts on the Hitchin fiber $M_{n,a}$.
If the spectral 
curve $\m{C}_a$ is smooth, then  the corresponding Hitchin  fiber $M_{n,a}\cong {\rm Pic}^{e'}(\m{C}_a)$ is smooth, and 
 a ${\rm Pic}^0(\m{C}_a)$-torsor via tensor product.
 \end{pr}
\begin{proof}
In view of the modular description
of $M_{n,a}$, it is clear that Fact \ref{norm}.(2)  implies that, for every $a \in A_n$, ${\rm Pic}^0(\m{C}_a)$ acts on $M_{n,a}$
via tensor product  (degree and stability are preserved), and that, when $\m{C}_a$ is smooth, this  action turns $M_{n,a}$ into the 
${\rm Pic}^0(\m{C}_a)$-torsor ${\rm Pic}^{e'}(\m{C}_a)$.
 Since the locus of characteristics in $A_n$ yielding a smooth spectral curve is open and dense in $A_n$, we conclude that $h_n$ is  dominant. Since $h_n$ is projective, it is  also surjective.
The same line of argument implies that the general fiber of $h_n$ is geometrically connected.
On the other hand, since $A_n$ is nonsingular, hence normal, Zariski main theorem implies that $h_n$ has 
geometrically connected fibers.
In view of  \ci[\S8,  Cor.]{ch-la}, the morphism  $h_n:M_n \to A_n$ is of pure relative
dimension the arithmetic genus of the spectral curves, which can be easily shown to be (\ref{rr}).
Since $M_n$ and $A_n$ are nonsingular, the pure-relative-dimension morphism $h_n$ is flat.
\end{proof}

\begin{rmk}\la{no} {\rm ({\bf No line bundles in the nilpotent cone when $(e,n)=1$})}
 The fiber $M_{n,0}$ over the origin does not contain line bundles. In fact, the spectral curve is
 of the form $nC$ (given by $t^n=0$ on the surface $V(D)$, a non-reduced curve with multiple structure  
 of multiplicity $n,$ and with  reduced curve $C;$ 
  it follows that every line bundle on it has degree ${\rm Deg}_{nC}$  a multiple of $n$; since the required degree is $e'= e+ {n \choose 2}d$ and $(e,n)=1$,  in general, there is no such line bundle (e.g. if $n$ is odd or if $d$ is even). By way of contrast,  if the spectral curve $\m{C}_a$
  is geometrically integral, then ${\rm Pic}^{e'}(\m{C}_a) \subseteq M_{n,a}$ is an integral,  Zariski-dense, open  subvariety.
  Finally, if we arrange for $e'=0$ (in which case, we may not have the coprimality of the pair $(e,n)$),
  one sees that, for every $a \in A_n$, the variety  ${\rm Pic}^0(\m{C}_a)$ is open  in $M_{n,a}$; see \ci[Cor. 5.2]{sc}.
  %{\cm  Moreover, every ${\rm Pic}^0(nC)$-orbit
%in $M_{n,0}$ has dimension strictly smaller than $\dim (M_{n,0})$, i.e. the corresponding stabilizers are all of positive &dimension. However, they are all affine by the forthcoming Proposition {\rm \ref{affstbz}}.}
 \end{rmk}

 {\bf Modular description of the Hitchin fiber $\check{h}_n^{-1}(a)$, $a\in \check{A}_n$.}
 The description in question is the same as the modular description given above, except for the  added constraint 
 on the determinant $\det ({p_a}_*\m{E}) = \e,$ where $\e \in {\rm Pic}^e(C)$ is the fixed line bundle 
 involved in the definition (\ref{hitm}) of $\check{M}_n$.
 
 \begin{defi}\la{defprym}
 {\rm ({\bf The Prym variety of a spectral cover})}
Let $a \in A_n$ and set
\beq\la{rb}
{\rm Prym}_{a}:= \ke \,\left\{ { N_{ p_a}}: {\rm Pic}^0 (\m{C}_a) \to 
{\rm Pic}^0(C_a)\right\}.
\eeq
\end{defi}
In general, the Prym variety ${\rm Prym}_a$  is a disconnected group scheme
with finitely many components; see \ci{ha-pa} for a description of these components at geometric points of $A_n$.
We also call Prym variety the corresponding  identity connected component. In a given context, we shall make it clear which
Prym variety we are using. 

If $a \in A_n$ is general,  then ${\rm Prym}_{a}$
is  geometrically connected (Fact \ref{norm}.(3)).

 \begin{pr}\la{kh}
The projective, $D$-twisted, $SL_n$ Hitchin morphism  $\check{h}_n: \check{M}_n \to \check{A}_n$ is surjective, with geometrically  connected fibers, flat of pure relative dimension
\beq\la{wdv}
d_{\check{h}_n} = d_{h_n} -g ={n\choose 2} d + (n-1)(g-1).
\eeq 
Let $a\in \check{A}_n$. Then  ${\rm Prym}_{a}$ acts on the Hitchin fiber $\check{M}_{n,a}$.
If the spectral 
curve $\m{C}_{a}$ is smooth, then ${\rm Prym}_{a}$ is connected,  the corresponding $SL_n$ Hitchin  fiber $\check{M}_{n,a}$ is smooth, and 
a ${\rm Prym}_{a}$-torsor via tensor product.
 \end{pr}
\begin{proof}
By Proposition \ref{kj},  for every $a\in \check{A}_n$, the $GL_n$ Hitchin fiber  $M_{n,a} \neq \emptyset.$
There is the natural morphism 
\beq\la{fr9}
\frak{p}_{a}:=\det \circ \, {p_{a}}_*: M_{n, a} \lorw {\rm Pic}^e(C).
\eeq
In view of the modular description of the $SL_n$ Hitchin fiber $\check{M}_{n, a}$, we have that
$\check{M}_{n, a}= {\frak p}_{a}^{-1} (\e)$.

\n
The morphism ${\frak p}_{a}$  is equivariant for the ${\rm Pic}^0(C)$-actions  given
by $L \cdot \m{E} := \m{E} \otimes L$ on the domain and by $L \cdot M:= M \cdot L^{\otimes n}$ on the target (Fact \ref{norm}.(2),(1)). It follows that ${\frak p}_{a}$ is surjective. In particular,  for every $a\in \check{A}_n$, $\check{M}_{n,  a} \neq \emptyset$,
so that $\check{h}_n$ is surjective.

\n
By Zariski main theorem, in order to check that $\check{h}_n$ has geometrically connected fibers, it is enough
to do so at a general point. We do this next.

\n
Since $\check{M}_{n, a}= {\frak p}_{a}^{-1} (\e)$, Fact \ref{kj}.(2) implies that
${\rm Prym}_{a}\subseteq {\rm Pic}^0(\m{C}_{a}) $ is the largest subgroup 
acting on $\check{M}_{n, a}.$ More precisely, if $\m{E} \in \check{M}_{n, a}$
and  $L \in  {\rm Pic}^0(\m{C}_{a})$, then $\m{E} \otimes \m{L}  \in \check{M}_{n, a}$
iff $\m{L} \in {\rm Prym}_{a}$.

\n
Let $a\in \check{A}_n$ be a  traceless characteristic yielding 
a nonsingular spectral curve $\m{C}_{a}$.  Since $M_{n, a}$ is a
${\rm Pic}^0(\m{C}_{a})$-torsor by Proposition \ref{kj},  we deduce that $\check{M}_{n,a}$ is s a
${\rm Prym}_{a}$-torsor.  For $a\in \check{A}_n$ general, ${\rm Prym}_{a}$  is geometrically connected by Fact \ref{norm}.(3),  then so  is the general fiber
  $\check{M}_{n, a},$ and, as anticipated, $\check{h}_n$ has thus  geometrically connected fibers. 
  
\n
Since all fibers of $\check{h}_n$ are  now known to be geometrically connected, so is the fiber 
$\check{M}_{n, a}$ corresponding to a smooth spectral curve. Since such a fiber is a ${\rm Prym}_{a}$-torsor,  the Prym variety ${\rm Prym}_{a}$ is also geometrically connected.

\n
Finally, 
since the morphism $\frak p_{a}$ is flat,  and the morphism $h_n$ is of pure dimension
(\ref{rr}), all the fibers of $\check{h}_n$ are of pure dimension $(\ref{rr})$ minus $g$, hence (\ref{wdv}) holds.
The  flatness of $\check{h}_n$ follows by this and  by the smoothness of $\check{M}_n$ and of $\check{A}_n.$
\end{proof}

%\begin{rmk}\la{endo}
%{\cm Even if $M_{n, \check{a}}$ is integral (e.g. when the spectral curve
%$\m{C}_a$  is integral, i.e. over the elliptic locus),  $\check{M}_{n ,\check{a}}$ may very well fail to be integral;
%this phenomenon is linked to endoscopy and to {\rm (\ref{rb})} not being connected. See \ci{ngoend,ha-pa}.}
%\end{rmk}

\subsection{Endoscopy loci of the Hitchin $SL_n$ fibration
}\la{vbn}
$\;$

%\subsection{The endoscopic locus $\check{A}_{n, {\rm endo}} \subseteq \check{A}_n$}\la{slnendo}$\;$

Let $a$ be a goemetric point of $A_n^{\rm ell}$, so that the spectral curve $\m{C}_a$ is (geometrically) integral. The $D$-twisted, 
$GL_n$ Hitchin fiber
$M_{n,a}$ is also  integral: it is isomorphic to the compactified Jacobian of the   integral locally planar  
spectral curve, parameterizing rank one and degree $e'$ torsion free coherent sheaves on it.
 In particular, the regular part ${\rm Pic}^{e'}(\m{C}_a) \cong M_{n,a}^{\rm reg} \subseteq M_{n,a}$ of this fiber is 
 integral, 
Zariski open and  dense in the whole fiber, and  it is a ${\rm Pic}^0(\m{C}_a)$-torsor.

Let  $a$ be a geometric point of $\check{A}_n$.  Then the $D$-twisted, $SL_n$ Hitchin fiber   $\check{M}_{n,a} =
{\frak p}_{a}^{-1}(\e)$  (cf. (\ref{fr9})), and it  is (geometrically) connected. 
Since the morphism $\check{h}_n$ is flat and $\check{A}_n$ is nonsingular, every fiber of $\check{h}_n$
is a local complete intersection (l.c.i).

Assume, in addition, that   $a$ is a geometric point of  $\check{A}_n^{\rm ell}$. 
By  the ${\rm Pic}^0(C)$-equivariance of $\frak p_{a}$, the regular part
of $\check{M}_{n,a}$   satisfies  $\check{M}_{n,a}^{\rm reg}= M_{n,a}^{\rm reg} \cap \check{M}_{n,a}$, and it is Zariski open and dense.  Since the fiber $\check{M}_{n,a}$ is a l.c.i., we have that, being smooth on a  Zariski-dense open subset, it is  also reduced. The regular part  $\check{M}_{n,a}^{\rm reg}$ is made of  line bundles $\m{E}$ on the spectral curve with ${\frak p}_{a} (\m{E})=\e$.
It is clear  that $\check{M}_{n,a}^{\rm reg}$ is then a ${\rm Prym}_{a}$-torsor.

\begin{fact}\la{num}
Let $a\in \check{A}_n^{\rm ell}$.
The discussion above implies  that  the number of irreducible components of the pure dimensional
and reduced $\check{M}_{n,a}$ 
coincides with the number of connected components of ${\rm Prym}_{a}$.
\end{fact}

For every $a\in A_n$, the group of connected components $\pi_0({\rm Prim}_{a})$ is described 
in \ci[Thm. 1.1]{ha-pa}.
The locus  $\check{A}_{n,{\rm endo}}\subseteq \check{A}_n$ over which  ${\rm Prym}_{a}$ is disconnected is called the endoscopic locus of the  $SL_n$ Hitchin fibration and it is  described
in \ci[\S5, especially Lemma 5.1;  Lemma 7.1]{ha-pa}:
\beq\la{340}
\check{A}_{n,{\rm endo}} = \bigcup_\Gamma \check{A}_{n,\Gamma},
\eeq
where: $\Gamma$ ranges over the finite set of  cyclic subgroups of ${\rm Pic}^0(C)[n]$
of prime number order. Each $\check{A}_{n,\Gamma} \subseteq \check{A}_n$ is  a geometrically integral 
subvariety.
The codimension   of each $\check{A}_{n,\Gamma}$ can be  computed in the same way as in the proof of \ci[Lemma 7.1]{ha-pa},
whose proof in the case   $D=K_C$, remains valid for $D$: we need 
the knowledge
of $d_{\check A_n}$  (\ref{ecce}), obtained by Riemann-Roch, and  the formula directly above \ci[Lemma 5.1)]{ha-pa}.
The resulting value
\beq\la{3ed}
{\rm codim}_{\check{A}_n}( \check{A}_{n,\Gamma}) =
\frac{1}{2} 
%\left(  r \right) \left[r \right] \left\{ r \right\}
 \left( n - \nu \right) \left\{  (n+\nu) d + \left[ d - 2(g-1)\right]  \right\}, \qquad (\nu:= n/\#(\Gamma)),
\eeq
is  strictly positive in view, for example,  of our assumption $d > 2(g-1)$. 

The subvarieties $\check{A}_{n,\Gamma}^{\rm ell}: = \check{A}_{n,\Gamma} \cap  \check{A}_n^{\rm ell}
\subseteq \check{A}_n^{\rm ell}$ 
are nonsingular and  mutually disjoint \ci[Prop. 10.3]{ngoend}.
By construction, the number  
\beq\la{e45}
o(\Gamma):=\# \left(\pi_0({\rm Prym}_{a})\right)
\eeq
 of connected components
of ${\rm Prym}_{a}$
is independent
of $a\in\check{A}_{n,\Gamma}^{\rm ell}$.

A point  $a\in \check{A}_{n,\Gamma}^{\rm ell}$ iff the spectral cover $p_a : \m{C}_{a} \to C$
has the property that the induced  morphism 
from the normalization of the integral spectral curve $\w{p_{a}}= \w{\m{C}_{a}} \to C$
factors through  the \'etale cyclic cover of $C$ associated with $\Gamma$ (cf \ci[Proof of Thm. 5.3]{ha-pa}).

The locus
\beq\la{endell}
\check{A}_{n,{\rm endo}}^{\rm ell} = \coprod_\Gamma  \check{A}_{n,\Gamma}^{\rm ell}
\eeq
is the $G=SL_n$ endoscopic locus introduced by Ng\^o in \ci[\S10]{ngoend} for  $D$-twisted, $G$ Hitchin fibrations ($G$ reductive). It determines the socle ${{\socle}} ({R{{\check h}_n}}_* \oql) \cap \check{A}_n^{\rm ell} $ over the elliptic locus; see \S\ref{zel}.

\subsection{Weak abelian fibrations and $\delta$-regularity}\la{w2}$\;$

The notion of $\delta$-regular weak abelian fibration has been introduced in \ci{ngofl} as an encapsulation of some
important features of the Hitchin fibration over the elliptic locus: presence of the  action 
of a commutative smooth group scheme with affine stabilizers, polarizability of the associated Tate module, $\delta$-regularity of the group scheme. See also  \ci{ngoaf} for an introduction to this circle of ideas.

In this section, let $g:J\to A$ be  a smooth commutative group scheme over  an irreducible  variety $A$ such that
$g$ has   geometrically connected fibers. 

{\bf Chevalley devissage.}
References for what follows are, for example:
  \ci[Thm. 10.25,  Prop. 10.24, Prop 10.5 (and its proof), Prop. 10.3]{milne} and \ci{bccd}.

Let $\ov{a}$  be  a geometric point on $A$ with underlying point a Zariski point  $a \in A.$
Let $J_{\ov{a}}$ be the fiber of $J$ at $\ov{a}$.
There is a canonical  short exact sequence of commutative connected  group schemes
 over the residue field of $\ov{a}$:
\beq\la{cbr}
0 \to J_{\ov{a}}^{\rm aff} \to J_{\ov{a}} \to J_{\ov{a}}^{\rm ab} \to 0,
\eeq
where $J_{\ov{a}}^{\rm aff} \subseteq J_{\ov{a}}$ is the maximal connected affine linear subgroup 
of $J_{\ov{a}}$, and $J_{\ov{a}}^{\rm ab}$ is 
an abelian variety. The dimensions of these varieties depend only on the Zariski point $a \in A,$ and are denoted by 
$d_a^{\rm aff} (J)$ and $d_a^{\rm ab} (J),$ respectively. Clearly, 
\beq\la{dzx}
d_a (J) = d_a^{\rm aff} (J) + d_a^{\rm ab} (J).
\eeq

{\bf The notion of $\delta$-regularity.}
The function
\beq\la{rf0}
\delta: A \lorw \zed^{\geq 0}, \qquad a \mapsto \delta_a: = d_a^{\rm aff}
\eeq
is upper semi-continuous (jumps up on closed subsets); see \ci[X, Rmk. 8.7]{sga3.2.2}.
We have the  disjoint union decomposition 
\beq\la{ds}
A= \coprod_{\delta \geq 0} S_{\delta}, \qquad S_\delta = S_\delta (J/A):= \left\{ a \in A\, |\; \delta_a = \delta\right\}
\eeq
of $A$ into  locally closed subvarieties of $A$. We call $S_\delta$ the $\delta$-locus of $J/A.$

\begin{defi}\la{delrg}   {\rm ({\bf $\delta$-regularity})}
We say that $g:J\to A$ is $\delta$-regular if 
\beq\la{bt}
{\rm codim}_A (S_\delta) \geq \delta, \qquad \forall \delta \geq 0,
\eeq
where one requires the inequality to hold for every irreducible component of $S_\delta$.
\end{defi}

The following lemma is an immediate consequence of the  upper-semicontinuity of the function $\delta$  and of 
the identity (\ref{dzx}).
\begin{lm}\la{gdr} A group scheme $g:J \to A$ as above is $\delta$-regular if and only if either of the two
 following equivalent conditions hold

\ben
\item
for every closed irreducible subvariety
 $Z \subseteq A$: let $\delta_Z$ be the minimum value of $\delta$ on $Z$
(it is attained at general points of $Z$, as well as at the generic point of $Z$); then ${\rm codim}_A (Z) \geq \delta_Z$;

\item
for every point $a \in A$, let $d_a:= \dim \ov{\{a\}}$; let $d_A:= \dim (A)$; then 
\beq\la{r4x}
d^{\rm ab}_a (J) \geq d_a(J) - d_A + d_a.
\eeq
\een
\end{lm}

{\bf The Tate module $T_{\oql}(J)$ and the notion of its polarizability.}
Let $g:J\to A$ be as above and let $d_g: = \dim{(J)} - \dim{(A})$ be the pure relative dimension
of $g$.
The Tate module of $J$ is the  $\oql$-adic sheaf (\ci[\S4.12]{ngofl})
\beq\la{tm}
T_{\oql}(J):= R^{2d_{g} -1} {g}_! \oql (d_g).
\eeq 
Its stalk at   any geometric point $\ov{a}$ of $A$ is given
by the Tate module $T_{\oql}(J_{\ov{a}})$, i.e. the inverse limit, with respect to $i\in \mathbb N$,  of the $\ell^i$-torsion points on $J_{\ov{a}}$,
tensored with $\oql$ over $\zed_\ell.$
The Chevalley devissage at the stalks yields  the natural short exact sequence  
\beq\la{zrt}
0 \to T_{\oql} (J^{\rm aff}_{\ov{a}})
 \to T_{\oql} (J_{\ov{a}})  \to T_{\oql}{(J^{\rm ab}_{\ov{a}})} \to 0.
 \eeq

The Tate module $T_{\oql}(J)$ is said to be polarizable if it admits a polarization,
i.e. an alternating bilinear pairing
\beq\la{pzx}
\psi: T_{\oql}(J) \otimes_{\oql} T_{\oql}(J) \lorw \oql (1),
\eeq
such that, for every geometric point  $\ov{a}$ of  $A$, we have that  the kernel of $\psi_{\ov{a}}$
is exactly $T_{\oql}(J^{\rm aff}_{\ov{a}})$. In this  case,  the pairings  $\psi_{\ov{a}}$ descend
to non-degenerate, alternating, bilinear parings on the $T_{\oql}(J^{\rm ab}_{\ov{a}}).$

Note that by general principles (cf. \ci[VIII, Cor. 4.10]{sga7.1}, the alternating bilinear pairings we consider in this paper
are automatically trivial on the ``affine" part, and do descend to the ``abelian" part. We do verify this fact along  the way to proving the key fact
that, in the cases we deal with, they  in fact descend to non-degenerate pairings.

{\bf Affine stabilizers.}
Let $h: M\to A$ be a morphism of varieties and let $J\to A$ be a group scheme   acting on  $M/A$.
We say that  that the action has affine stabilizers if for every geometric  point $m$ of $M$, we have that
the stabilizer subgroup ${\rm St_m} \subseteq J_{h(m)}$  is affine.

{\bf $\delta$-regular weak abelian fibrations.}
See \ci{ngofl,ngoaf}. 
Let $h: M\to A \leftarrow J:g$ be a pair of morphisms of varieties, where  
$g$ is  as in in the beginning of this section  \S\ref{w2} (smooth commutative group scheme, with geometrically connected fibers over  an irreducible $A$), $h$ is proper, and $J/A$ acts on $M/A$. We denote this situation
simply by $(M,A,J)$; the context will make it clear which morphisms $h,g$ are being used.

\begin{defi}\la{waf}
{\rm ({\bf Weak abelian fibration})}
We say that $(M,A,J)$ is a weak abelian fibration if $g$ and $h$ have the same pure relative dimension,
the Tate module $T_{\oql}(J)$ is polarizable and the action has affine stabilizers.
{\rm ({\bf $\delta$-regular weak abelian fibration})} A weak abelian fibration $(M,A,J)$ is said to be $\delta$-regular if  $g:J\to A$ is $\delta$-regular
as in Definition {\rm \ref{delrg}}, equation {\rm (\ref{bt})}, or equivalently as in Lemma 
{\rm \ref{gdr}}, equation {\rm (\ref{r4x})}.
\end{defi}

{\bf Ng\^o support inequality.}
The following is a remarkable, and remarkably useful, topological restriction on the dimensions of the supports appearing in the context of weak abelian fibrations. If $a\in A$, then 
$d_a:= \dim{\ov{\{ a\}}}$ is the dimension of the closed subvariety of $A$ with generic point $a$.
For the notion of socle, see \S\ref{intro}.  The celebrated Ng\^o support theorem \ci[Thm. 7.2.1]{ngofl} is a more refined
 restriction
on the geometry of the supports, and it  is proved also by using the support inequality.

  \begin{tm}\la{nst} {\rm  ({\bf N\^go's support  inequality} \ci[Thm. 7.2.2]{ngofl})}
Let $(M,A,J)$ be a weak abelian fibration with $M$ and $A$  nonsingular  and with $h$ projective of pure relative dimension
$d_h$.
If  $a \in {\socle}  (Rh_* \oql)$, then: 
\beq\la{nin}
d_h - d_A + d_a \geq d_a^{\rm ab}(J).
\eeq
\end{tm}
Given that we are assuming $d_h=d_g$, we may re-formulate (\ref{nin}) as follows via (\ref{dzx})
\beq\la{xcd}
d^{\rm aff}_a(J) \geq  {\rm codim} (\ov{\{a\}} ).
\eeq

\section{The $GL_n$ weak abelian fibration}\la{drwaf}$\;$

This \S\ref{drwaf} is devoted to a detailed  study of the $\delta$-regular weak abelian fibration  $(M_n,A_n, J_n),$
arising from the action of  the Jacobi group scheme $J_n/A_n$, associated with the family of spectral curves of the $GL_n$ Hitchin fibration $M_n/A_n.$
\S\ref{jn} introduces the  Jacobi group scheme $J_n/A_n$ and its action on $M_n/A_n$: its fibers are the Jacobians
of the spectral curves.
\S\ref{affstb} shows that the stabilizers for this action are affine.   I am  not aware of an explicit reference in the literature
for this  result  over the whole base $A_n$; \ci[4.15.2]{ngofl} deals with  a suitable open proper subset of $A_n$,
and for every $G$ reductive. \S\ref{poltm} is devoted to the lengthy proof that the Tate module
associated with $J_n/A_n$ is polarizable over the whole base $A_n$. Again,  I am  not aware of an explicit reference in the literature
for this result  over the whole base $A_n;$ the standard reference for this important-for-us
technical fact is \ci[\S4.12]{ngofl}, which deals with the situation over the  elliptic locus
$A_n^{\rm ell} \subseteq A_n.$ Following  this preparation, \S\ref{00}
contains the main result of this section, namely Theorem \ref{f4}, to the effect that 
$(M_n,A_n, J_n)$ is a weak abelian fibration that is $\delta$-regular over the elliptic locus $A_n^{\rm ell};$
this affords the    support inequality over the whole $A_n$, and the  $\delta$-regularity inequality
over the elliptic locus $A_n^{\rm ell}$. We need some of these explicit  details  of this $GL_n$ section  \S\ref{drwaf},
especially in connection with non-reduced spectral curves,
in view of our main Theorem \ref{maintm} on the $SL_n$ socle.

\subsection{The action of the Jacobi group scheme $J_n$}\la{jn}$\;$

For what follows, see \ci[\S5]{ch-la}.
Let $J_n \to A_n$ be the identity connected component of the degree zero component of the relative Picard stack
${\rm Pic}_{\m{C}/A_n}$.
This is a connected, smooth, commutative group scheme over $A_n$, whose fiber $J_{n,a}$ over a point $a \in A_n$
is  ${\rm Pic}^0 (\m{C}_a)$; see Fact \ref{piczero} for a description of this group.  In particular, the structural morphism $g_n:J_n \to A_n$ is of pure relative dimension, call it $d_{g_n}$, the arithmetic genus of the spectral curves,
which coincides with the pure relative dimension   $d_{h_n}$   (\ref{rr}) of $h_n: M_n \to A_n$, i.e. we have
\beq\la{oim}
d_{g_n} = d_{h_n}.
\eeq

The group scheme $J_n/A_n$ acts on the Hitchin fibration $M_n/A_n;$ see Proposition \ref{kh}.

\subsection{Affine stabilizers  for the action of  the Jacobi group scheme}\la{affstb}$\;$

\begin{pr}\la{affstbz}
The action of $J_n/A$ on $M_n/A_n$ has  affine stabilizers.
\end{pr}
\begin{proof}
Let $a$ be a geometric point of $A_n$ and let $\m{E} \in M_{n,a}$.
Recall that $\rm Rk_{\m{C}_a} (\m{E})=1$ means that,  with the notation of \S\ref{spcv}, if $\m{C}_a = \sum_k m_k \Gamma_k$,
with $m_k \geq 1$ for every $k$, 
then the length of $\m{E}$ at the stalk of the generic point of $\Gamma_k$ is $m_k,$ for every $k.$ 
Let $\xi: \w{\m{C}_{a,{\rm red}}} = \coprod_k \w{\Gamma_k}  \to \m{C}_a$ be the morphism from the  normalization of $\m{C}_{a, {\rm red}}$ (cf. (\ref{gh})).
Let $0 \to {\rm Tors} (\xi^* \m{E}) \to \xi^* \m{E} \to \xi^* \m{E}/{\rm Tors} (\xi^* \m{E})=:\ms{E} \to 0$ be the canonical
devissage of the torsion  of $\xi^* \m{E}$ on the nonsingular   projective curve   $\coprod_k \w{\Gamma_k}$.
Let $\m{L} \in {\rm Pic}^0(\m{C}_a)$. Assume that $\m{L}$ stabilizes $\m{E}$. Then $\xi^* \m{L}$
stabilizes every term in the canonical torsion devissage of $\xi^*( \m{E}) \otimes \xi^*\m{L}$.
In particular, $\xi^*\m{L}$ stabilizes the vector bundle $\ms{E}$, which has rank $m_k$ on each $\w{\Gamma_k}$.
 By considerations of determinants, we see that $\xi^* \m{L}
\in \prod_k {\rm Pic}^0(\w{\Gamma_k}) [m_k]$, a finite group.
The natural morphism $\xi^*: {\rm Pic}^0(\m{C}_{a,{\rm red}}) \to {\rm Pic}^0(\w{\m{C}_a})=
\prod_k {\rm Pic}^0(\w{\Gamma_k})$ is surjective, with affine (connected) kernel (cf. \ci[\S9]{neron}).
It follows that the stabilizer of $\m{E}$ is an extension of a finite group by an affine subgroup, so that it is affine.
\end{proof}

\subsection{The Tate module of  the  Jacobi group scheme  is polarizable  }\la{poltm}$\;$

We refer to \S\ref{w2} for the terminology.
Let $g_n: J_n \to A_n$ be the structural morphism for Picard. The Tate module  is the  $\oql$-adic sheaf (\ref{zrt})
$T_{\oql}(J_n):= R^{2d_{h_n} -1} {g_n}_! \oql (d_{h_n})$. If $a$ is a geometric point of $A_n,$ then the Chevalley devissage
 yields  the natural short exact sequences
 \beq\la{zr}
0 \to T_{\oql}(J^{\rm aff}_{n,a})
 \to T_{\oql} (J_{n,a})  \to T_{\oql}{(J^{\rm ab}_{n,a})} \to 0.
 \eeq
 Note that:  (i) $\dim_{\oql} T_{\oql}(J^{\rm ab}_{n,a}) = 2 \dim J^{\rm ab}_{n,a}$; (ii) 
  $\dim_{\oql} T_{\oql}(J^{\rm aff}_{n,a}) \leq \dim J^{\rm aff}_{n,a}$, and that the strict inequality
 can occur:  this is due to the fact that
 the affine part  $J^{\rm aff}_{n,a}$ is an iterated extension  of   the additive and of the multiplicative group $\mathbb G_a$ and $\mathbb G_m$
 \ci[\S9]{neron}, and only the latter contribute
 to the Tate module.

The goal of this section is to prove the following polarizability result, which has been proved 
over the elliptic locus $A_n^{\rm ell}$ in \ci{ngofl}, and is stated implicitly over the whole base $A_n$ and then used in \ci[\S9]{ch-la}.

\begin{tm}\la{poltmd}
The Tate module $T_{\oql}(J_{n})$  on $A_n$ is polarizable.
\end{tm}
\begin{proof}  
Let $p:\m{C} \to A_n$ be the family of spectral curves: it is proper, flat, with geometrically connected fibers,
with nonsingular total space, and with nonsingular general fiber. As in \ci[\S4.12]{ngofl},
the pairing is defined by constructing it over the strict henselianization of the local ring of  any Zariski point $a \in A_n$,
for the construction yields a canonical outcome. We denote these new shrunken  families  by $p:\m{C} \to A$, $g: J\to A$.
For a coherent sheaf $F$ on $\m{C}$, set
$\Delta (F):= \det (Rp_* F),$ where we are taking the determinant of cohomology \ci[especially, \S1.4]{deligne,soule} and the result is a graded line bundle on $A$. If $F$ is $\m{O}_A$-flat, then  the  degree of this graded line bundle  is the Euler characteristic of $F$ along the fibers $\m{C}_a.$
The Weil pairing construction associates with  $L, M \in {\rm Pic}^0(\m{C}/A)$
the graded line bundle on $A$ given by the formula
\beq\la{dfre}
\langle L, M\rangle_{\m{C}/A}:= P(L,M): =\Delta (L\otimes M) \otimes  \Delta (\m{O}_A) \otimes 
 \Delta (L)^\vee \otimes  \Delta ( M)^\vee.
\eeq
Note that both of the terms defined by (\ref{dfre}) make sense for  any pair of coherent sheaves on $\m{C}$, 
however, we shall use $\langle -,- \rangle$ when dealing with line bundles, whereas
we shall use $P(-,-)$ also for other coherent sheaves, hence the two distinct pieces of  notation.

\n
Let $L, M \in {\rm Pic}^0(\m{C}/A) [\ell^i]$ be $\ell^i$-torsion line bundles.
The formalism of the determinant of cohomology
yields two distinguished  isomorphisms $i_L, i_M:\langle L, M\rangle^{\otimes \ell^i}_{\m{C}/A} \lorw \m{O}_S$
%(trivial, trivially trivialized line bundle in degree zero)
whose difference $\e_{L,M}$ is an $\ell^i$-th root of unity in the ground field and which depends only on the isomorphism classes of $L$ and of $M$.
By taking inverse limits with respect to $i$, and then by  tensoring  with $\oql$, we obtain  a pairing, let us call it the  Tate-Weil pairing
\beq\la{tw}
TW: T_{\oql}(J) \otimes_{\oql} T_{\oql}(J) \lorw T_{\oql} ({\mathbb G}_m) = \oql (1),
\quad \{L_i, M_i\}_{i \in \mathbb N} \mapsto \{\e_{L_i,M_i}\}_{i \in \mathbb N} \in \zed_\ell (1).
\eeq
The Weil and the Tate-Weil pairing are compatible with base change.

\n
Let $a $ be a geometric point of $A$, consider the diagram (\ref{gh}) of maps of curves, and extract the 
following morphisms
\beq\la{4rf}
\xymatrix{
\xi = \coprod_k \xi_k:
\coprod_k \w{\Gamma_k} \ar[rr]^{\hskip 9mm \xi_{3}= \coprod_k \xi_{3,k}} &&  \coprod_k \Gamma_k  
\ar[rr]^{\hskip -3mm \xi_{4}= \coprod_k \xi_{4,k}} && \sum_k m_k \Gamma_k,
}
\eeq
\beq\la{pax}
\xi:\w{\m{C}_{a,{\rm red}}}  \stackrel{\nu}\lorw  {\m{C}}_{a,{\rm red}} \stackrel{\rho}\lorw \m{C}_a,
\qquad
\xi = \coprod_k \xi_k:  \coprod_k \w{\Gamma_k} \stackrel{\nu= \coprod_k \nu_k}\lorw  \sum_k \Gamma_k
 \stackrel{\rho}\lorw \sum_k m_k \Gamma_k.
\eeq
{\bf Claim.}
Let $L,M \in J_a={\rm Pic}(\sum_k m_k \Gamma_k).$ 
Then
\beq\la{wuf}
\langle L, M \rangle_{\sum_k m_k \Gamma_k} =
\bigotimes_k \langle \xi^*_{4,k} L, \xi^*_{4,k} M \rangle_{\Gamma_k}^{\otimes m_k} = 
\bigotimes_k \langle \xi_k^*L, \xi_k^* M \rangle_{\w{\Gamma_k}}^{\otimes m_k}.
\eeq
In order to prove this claim, we first  list the three short exact sequences below.

\n
The ideal sheaf of $\m{C}_a$ in $\m{O}_{V(D)\otimes k(a)}$ is locally generated by the product  $\prod_{k=1}^s \frak{s}_k^{m_k}$ (cf. \S\ref{spcv}) of powers  of sections of the line bundle $\pi^*D$ 
 on the surface $V(D)\otimes k(a)$. Fix   any index $1 \leq k_o \leq s$; fix any sequence $\{\mu_k\}_{k=1}^s,$ with 
 $0 \leq \mu_k \leq m_k$ for every $k$, with 
 $1\leq \mu_{k_o}$, and with $\sum_k \mu_k \geq 2$ (these conditions are simply to ensure that
 (\ref{rfty}) below is meaningful as written).  We have the following
system of short exact sequences  on the curve $\sum_k \mu_k \Gamma_k$
(see \ci[Lemma 3.10]{mreid},  for example)
\beq\la{rfty}
0 \lorw  \m{O}_{\Gamma_{k_o}} (-\Gamma_{k_o}) \lorw 
\m{O}_{\sum_k \mu_k \Gamma_k} \lorw  \m{O}_{(\mu_{k_o}-1)\Gamma_{k_o} + \sum_{k\neq k_o}\mu_k \Gamma_k}
\lorw 0.
\eeq
We have the short exact sequences (\ref{rv0}) on the curves $\Gamma_k$
\beq\la{rft}
0 \lorw  \m{O}_{\Gamma_{k}} (-\Gamma_{k}) \lorw  
\m{O}_{\Gamma_k} \lorw  \m{O}_{\zeta_k}
\lorw 0.
\eeq
We  have  a natural short exact sequence on $\coprod_k \Gamma_k$ arising from the normalization map $\xi_3$
\beq\la{rf}
0 \lorw  \m{O}_{\coprod_k \Gamma_{k}}  \lorw 
\m{O}_{\coprod_k\w{\Gamma_k}} \lorw  \Sigma
\lorw 0,
\eeq
where $\Sigma$ is supported at finitely many points on $\sum_k \Gamma_k.$

\n
Since $\Sigma$ is supported at finitely many points, it follows
from the definition that, for every pair of line bundles $L,M$ on the curve $\sum_k \Gamma_k$, we have that 
$P(\Sigma\otimes L, \Sigma\otimes M)$ is canonically isomorphic to the trivially trivialized,
trivial line bundle on the spectrum of the residue field of $a;$ see \ci[proof of Lemma 4.12.2]{ngofl}.
We call this circumstance, the $P$-triviality property
of $\Sigma$. 
The same holds true for $P(\m{O}_{\zeta_k} \otimes L, \m{O}_{\zeta_k} \otimes M)$, i.e. we have
the $P$-triviality property of  $\zeta_k$.

\n
By what above, and by using the multiplicativity property of the determinant of cohomology with respect to  short exact sequences, and hence of the
operation $P(-,-)$, 
 we see that the second equality of the claim (\ref{wuf}) follows from the short exact sequence
(\ref{rf}) on $\coprod_k \Gamma_k$, by using the $P$-triviality property of $\Sigma,$ and the fact that $\xi^*_k = \xi_{3,k}^*\circ \xi_{4,k}^*$; in fact,   we get,
the identity
\[
P(\xi_{3,k}^*  \xi_{4,k}^*, \xi_{3,k}^*  \xi_{4.k}^*M) = P( \xi_{4,k}^* L,  \xi_{4,k}^*M)  \otimes P(\Sigma \otimes 
\xi_{4,k}^*L, \Sigma \otimes \xi_{4,k}^*M) = P(\xi_{4,k}^*L,\xi_{4,k}^*M).
\]
(N.B. there is no need for the exponents $m_k$, 
for this second equality in (\ref{wuf}).)

\n
The first equality of the claim (\ref{wuf}), and here the exponents $m_k$ are essential, follows in the same way (by using 
the $P$-triviality property for $\zeta_k$) from (\ref{rfty}) and 
(\ref{rft}) by means of a simple  descending induction on the multiplicities $\mu_k \leq m_k$, based on the 
following equalities (where we denote line bundles and their restrictions in the same way, and we instead  add a subfix
to $P(-,-)$)
\[
P_{ \sum_{k} \mu_k \Gamma_k } (L,M) = P_{(\mu_{k_o}-1)\Gamma_{k_o}+ \sum{k\neq k_o} \mu_k \Gamma_k}
(L,M)
\otimes P_{\Gamma_{k_o}} (L-\Gamma_{k_o}, M- \Gamma_{k_o}),
\]
\[
P_{\Gamma_{k_o}} (L-\Gamma_{k_o}, M- \Gamma_{k_o}) = P_{\Gamma_{k_o}} (L, M)
\otimes P_{\zeta_{k_o}} ( L,  M)= P_{\Gamma_{k_o}} (L, M).
\]
We now use the just-proved claim (\ref{wuf}) to verify that the Tate-Weil pairing $TW$ (\ref{tw}) has,
at every geometric point $a$ of $A_n$,  kernel given {\em precisely}  by
the ``affine part"  $T_{\oql} (J_{n,a}^{\rm aff})$, so that it descends to a non-degenerate pairing
on $T_{\oql} (J_{n,a}^{\rm ab})$.

\n
By \ci[\S9.3, Corollary 11]{neron}, we have  the canonical  short exact sequence
\beq\la{q0}
0 \lorw \ke \,\xi^* \lorw   J_{n,a}= {\rm Pic}^0 \left(\m{C}_a=\sum_k m_k \Gamma_k\right) 
\lorw
{\rm Pic}^0(\w{\m{C}_{a,{\rm red}}}) =
\prod_k {\rm Pic}^0\left(\w{\Gamma_k}\right)  \lorw 0,
\eeq
with  quotient an abelian variety and with  affine and {\em connected}  $\ke\, \xi^*$, an iterated extension
of groups of type $\mathbb G_a$ and $\mathbb G_m$. It follows that the above short exact sequence {\em is}
the ``abelian-by-affine" Chevalley devissage  (\S\ref{w2}) of $J_{n,a}$.
By passing to Tate modules, we get the short exact sequence
\beq\la{qw}
0\lorw
T_{\oql} (\ke \,\xi^*) = T_{\oql} (J_{n,a}^{\rm aff}) \lorw T_{\oql}(J_{n,a}) \lorw T_{\oql} (J_{n,a}^{\rm ab}) = \bigoplus_k T_{\oql} ({\rm Pic}^0(\w{\Gamma_k}))  \lorw 0.
\eeq
%Note that since $\ke {\, \xi}^*$ is an iterated extension of groups of type $\mathbb G_a$ and of type 
%$\mathbb G_m$, the rank of the Tate module $T_{\oql} (\ke {\, \xi})$ can very well be strictly
%smaller than the dimension of $\ke {\, \xi^*}$; this discrepancy, plays no role in what follows.

\n
In view of (\ref{wuf}), and of the definition of the Tate-Weil   pairing via the Weil pairing, we see that
the kernel of the  Tate-Weil pairing contains $T_{\oql} (\ke \,\xi^*)$, so that the Tate-Weil  
pairing $TW:=TW_{\sum_k m_k \Gamma_k}$   descends
to  a paring $TW^{\rm ab} $ on   the abelian part $T_{\oql} ( J_{n,a}^{\rm ab})$ where, again in view of (\ref{wuf}),  it is the direct sum 
of the Tate-Weil pairing $TW_{\w{\Gamma_k}}$ on the individual nonsingular  projective curves $\w{\Gamma_k}$,
{\em multiplied} by the integer $m_k$ 
\beq\la{q10}
TW^{\rm ab}=
\sum_k m_k TW_{\w{\Gamma_k}}.
\eeq
Each $TW_{\w{\Gamma_k}}$ is non-degenerate: in fact, it is  the   Tate-Weil pairing on the 
Tate module of the Jacobian of a nonsingular
projective curve over an algebraically closed field, which, in turn,  can be identified with the cup product
on the first \'etale  $\oql$-adic cohomology group  of the curve; see \ci[Ch. V, Remark 2.4(f), and references therein]{miec}. It follows that their $m_k$-weighted direct sum $TW^{\rm ab}$ is non-degenerate as well.
\end{proof}

\begin{rmk}\la{tapt}   \ci[\S9.2, Thm. 11]{neron} gives a precise structure theorem for the  Jacobians of curves which immediately yields the following description of their abelian variety parts.
Let  $a$ be  a geometric point of  $A_n$ and let $\cu{C}{a} = \sum_k m_k \Gamma_k$ be corresponding spectral curve.
Then we have natural isomorphisms of abelian varieties
\beq\la{3030}
{\rm Pic}^0(\cu{C}{a})^{\rm ab} = {\rm Pic}^0(\cu{C}{a, {\rm red}})^{\rm ab}
=\prod_{k} {\rm Pic}^0(\Gamma_k)^{\rm ab} = \prod_{k} {\rm Pic}^0(\w{\Gamma_k}).
\eeq
\end{rmk}

\subsection{$\delta$-regularity of the action of the Jacobi group scheme over the elliptic locus}\la{00}$\;$

\begin{tm}\la{f4}
The triple $(M_n,A_n,J_n)$ is a weak abelian fibration and its restriction over $A_n^{\rm ell}$ is
a $\delta$-regular weak abelian fibration.
In particular,
\ben
\item
If $a \in {\socle}(R{h_n}_* \oql)$, then 
\beq\la{99}
d_{h_n} - d_{A_n} + d_a \geq d_a^{\rm ab}(J_n)  \quad \mbox{\rm ({\bf Ng\^o support inequality})}.
\eeq
\item
If $a \in A_n^{\rm ell}$, then 
\beq\la{98}
d_a^{\rm ab}(J_n) \geq d_{h_n} - d_{A_n} + d_a \quad \mbox{\rm ({\bf $GL_n$ $\delta$-regularity inequality})}.
\eeq
\een
\end{tm}
\begin{proof} The two morphisms $h_n$ and $g_n$ have the same pure relative dimension
(\ref{oim}).
The stabilizers of the action are affine by Proposition \ref{affstbz}. The Tate module is polarizable
by Theorem \ref{poltmd}. It follows that the triple is indeed a weak abelian fibration.
Since $h_n$ is projective and $M_n$ is nonsingular,   (\ref{99})
follows from Ng\^o support inequality Theorem \ref{nst}.
The inequality (\ref{98}) is known as 
``Severi's inequality", see \ci[Thm. 7.3]{ch-la}, for references; see also 
\ci[the paragraph following Thm. 2 on p.3]{fa-go-va}.
The $\delta$-regularity assertion (\ref{98}) then follows from Lemma \ref{gdr}, equation (\ref{r4x})).
\end{proof}

  \section{The $SL_n$ weak abelian fibration}\la{s3}$\;$

\S\ref{s3} is devoted to proving Theorem \ref{95}, i.e.  the $SL_n$ counterpart to Theorem \ref{f4} for $GL_n$.
\S\ref{jnc} introduces the group scheme $\check{J}_n/\check{A}_n$ of identity components
of the   Prym group scheme, which, in turn, has fibers (\ref{rb}) that  become disconnected precisely over the endoscopic locus
(\ref{340}).
\S\ref{sw3}  establishes the precise relation between the abelian-variety-parts of the fibers 
of the Jacobi group scheme $J_n/A_n,$ and the ones  of the Prym-like group scheme $\check{J}_n/\check{A}_n$; this is a key step
in establishing the $\delta$-regularity of $\check{J}_n$ over the elliptic locus.
\S\ref{prd} establishes the expected  product  structure of $M_n$, with factors $M_n(0)$ (traceless Higgs bundles)
and $H^0(C,D)$ (space of possible traces); this is another key step towards the $\delta$-regularity above.  
These factorizations are further pursued in \S\ref{lm1r}, where one factors $J_n$ in the same way.
\S\ref{9a} establishes the  $\delta$-regularity of $\check{J}_n/\check{A}_n$ over the elliptic locus
$\check{A}_n^{\rm ell}.$
\S\ref{tnmq} studies in detail the norm morphism associated with arbitrary (not necessarily irreducible, nor reduced) spectral curves.
\S\ref{itpz} establishes the key polarizability of the Tate module of $\check{J}_n$ over the whole base 
$\check{A}_n$ by using: the  explicit form  (\ref{q10}) of the polarization of the Tate module of $J_n$;
the explicit form (\ref{eq01z}) of the norm map; a formal reduction of the  $SL_n$ polarizability result
to the classical fact that, at the level of Tate modules of Jacobians, the maps induced by the pull-back and by the norm are adjoint for the Tate-Weil pairing.
\S\ref{33}  is devoted to binding-up the results of this section by establishing
Theorem \ref{95}, i.e. the $SL_n$ counterpart to  Theorem \ref{f4} for $GL_n,$ to the effect that 
$(\check{M}_n, \check{A}_n, \check{J}_n)$ is a weak abelian fibration
which is $\delta$-regular over the elliptic locus; this yields the support inequality over the whole base $\check{A}_n$, and the $\delta$-regularity  inequality over the elliptic locus $\check{A}_n^{\rm ell}.$
\S\ref{zel} is devoted to spelling-out the supports for the $SL_n$ Hitchin fibration
over the elliptic locus $\check{A}_n^{\rm ell}$; the results over the elliptic locus in this \S\ref{zel}, and for every $G$, are due to B.C. Ng\^o \ci{ngofl}.

\subsection{The action of the Prym group scheme $\check{J}_n$}\la{jnc}$\;$

Let $p: \m{C} \to A_n$ be the family of spectral curves as in \S\ref{spcv}.
 The norm morphism (\ref{nme})  defines  a morphism of group schemes over $A_n$
 (cf.  \ci[Cor. 3.12]{ha-pa}, for example)
\beq\la{nm}
N_p: J_n \lorw  {\rm Pic}^0(C) \times A_n, \qquad L \mapsto \det (p_*L) \otimes [\det (p_*( \m{O}_{\m{C}}))]^{-1}.
\eeq

The $A_n$-morphism  $p: \m{C} \to A_n$ induces the morphism  $p^*: {\rm Pic}(C)\times A \to
J_n$ of group schemes over $A_n$. One verifies that $N_p (p^* (-) )= (-)^{\otimes n};$
see Fact \ref{norm}.(1). In particular, the morphism $N_p$ is surjective. The differential 
of the composition $N_p \circ p^*$ along the identity section 
is multiplication by $n$, so that the morphism $N_p$ is smooth.  
%In particular, if $\char k >n$,  the morphism $N_p$ is surjective and smooth:
%surjective because ${\rm Pic}^0(C)$ is $n$-divisible
%in fact, its source and target are smooth, it is surjective on closed points and  all fibers have the same dimension, so that $%N_p$ is flat,
%and the fibers are smooth group schemes, so that the morphism is smooth. 
The kernel $\ke (N_p)$ of 
$N_p$  is a closed subgroup scheme that is smooth over $A_n.$ We call it the Prym group scheme.   Its fibers are precisely the Prym
varieties (\ref{rb}). 
Then,  by \ci[Exp VI-B, Thm. 3.10]{sga3.1}, there is the  open subgroup scheme over $A_n$ 
\beq\la{er4}
J'_n:= (\ke (N_p))^0
\eeq
 of the kernel, which (set-theoretically) is the union of the identity  connected components of the  fibers of this kernel group scheme over $A_n$. Since this whole construction is compatible with arbitrary base change,    the fiber  $J'_{n,a}$ over $a \in A_n$ is precisely the identity connected component of the kernel of the norm morphism associated with the spectral cover $\m{C}_a \to C_a =C\otimes k(a).$

We restrict this whole picture to the $SL_n$ Hitchin base $\check{A}_n = A_n(0) \subseteq A_n$
and set 
\beq\la{ert}\check{J}_n:= {J'_n}_{|\check{A}_n},\eeq
which we also call the Prym group scheme.

Then $\check{J}_n/\check{A}_n$ is a smooth connected group scheme with connected fibers  over $\check{A}_n$
that acts on $M_n(0)/A_n(0)$   (trace zero) preserving $\check{M}_n/\check{A}_n$ (trace zero and fixed determinant $\e$);
see Fact \ref{norm}.(3) and the proof of Proposition \ref{kh}.
It follows that $\check{J}_n/\check{A}_n$ acts on $\check{M}_n/\check{A}_n.$

According to Proposition \ref{kh},
on each fiber $\check{M}_{n,a}$, this action is
free  on the open  part given by those rank one torsion free sheaves   which are locally free.
The Hitchin fibers $\check{M}_{n,a}$ corresponding to nonsingular spectral curves
are $\check{J}_{n,a}$-torsors via this action.

\subsection{The   abelian variety  parts}\la{sw3}$\;$

Let $a$ be a geometric point
of $A_n$ ($\check{A}_n$, resp.).
Recalling the Chevalley devissage \S\ref{w2} for $J_{n,a}$ ($\check{J}_{n,a}$, resp.),  we set, by taking dimensions as varieties over the algebraically closed residue field of $a$
\[
d^{\rm ab}_a (J_n):= \dim (J_{n,a}^{\rm ab}), \qquad \check{d}^{\rm ab}_a (\check{J}_n):= \dim (\check{J}_{n,a}^{\rm ab});
\]
these dimensions depend only on the Zariski point underlying $a$.

%NOTE FOR SELF 
%The following easy lemma is used in the proof of Theorem \ref{maintm}  in \S\ref{pmaintm}; see (\ref{eq02}).
 \begin{lm}\la{azz}
 For every point $a \in \check{A}_n$, we have that:
 \beq\la{eq0}
 d^{\rm ab}_a(\check{J}_n) \geq d_a^{\rm ab}(J_n)-g.
 \eeq
 \end{lm}
 \begin{proof}
 Since $J^{\rm aff}_{n,a}$ is the biggest affine normal connected group subscheme inside $J_{n,a}$,
 we must have $d^{\rm aff}_a (\check J_n)  \leq d^{\rm aff}_a (J_n).$
Since $\dim{(J_{n,a})} = \dim{(\check{J}_{n,a})} + g,$ the conclusion follows.
\end{proof}

In fact, as Proposition \ref{azz=} below  shows, the inequality of Lemma \ref{azz} is an equality. 
% NOTE FOR SELF
%We  need this 
%equality in the proof of the polarization Theorem  \ref{slnpol}, which allows us to deduce the support inequality %Theorem \ref{95}, which, finally, is also used in the proof of  Theorem \ref{maintm} given  \S\ref{pmaintm}; see 
%(\ref{eq01}).

\begin{pr}\la{azz=}
 For every geometric point $a$ of  $\check{A}_n$, we have that:
 \beq\la{eq000}
 d^{\rm ab}_a(\check{J}_n) = d_a^{\rm ab}(J_n)-g. 
 \eeq
 More precisely, we have 
 \beq\la{lk}
 \check{J}_{n,a}^{\rm aff} = J_{n,a}^{\rm aff} \subseteq J_{n,a},
\eeq
\beq\la{jh} 
 J_{n,a}/\check{J}_{n,a} \cong J^{\rm ab}_{n,a}/\check{J}^{\rm ab}_{n,a},
 \eeq
 and a natural isogeny   
 \beq\la{hg}
  J^{\rm ab}_{n,a}/\check{J}^{\rm ab}_{n,a} \lorw {\rm Pic}^0(C_a).
  \eeq
 \end{pr}
 \begin{proof}
 Recall that 
we have the surjective norm morphism $N_p: J_{n,a} = {\rm Pic}^0(\m{C}_a) \to {\rm Pic}^0(C_a)$ and that
 $\check{J}_{n,a}:= (\ke{(N_p)})^0.$ We thus obtain the natural isogeny $J_{n,a}/\check{J}_{n,a} \to {\rm Pic}^0(C_a)$.
 In particular, $J_{n,a}/\check{J}_{n,a}$ is an abelian variety of dimension $g.$
  
 \n
 In view of the Chevalley devissage construction, we
 have the commutative diagram of short exact sequences of morphisms
 \beq\la{bdgr}
 \xymatrix{
 0 \ar[r]  & \check{J}^{\rm aff}_{n,a} \ar[r] \ar@{.>}[d]^u &  \check {J}_{n,a}  \ar[r] \ar[d]^v &    \check{J}^{\rm ab}_{n,a} \ar[r] \ar@{.>}[d]^w & 0 \\
  0 \ar[r]  & J^{\rm aff}_{n,a} \ar[r]  &  J_{n,a}  \ar[r]  &    J^{\rm ab}_{n,a} \ar[r]  & 0,
  }
 \eeq
 where: $v$ is the natural inclusion; $u$, also an inclusion, arises from the fact
 that in the Chevalley devissage, $J^{\rm aff}_{n,a}$ is the biggest connected affine subgroup of $J_{n,a}$,
 so that it contains all other connected affine subgroups of $J_{n,a}$, so that it contains $\check{J}^{\rm aff}_{n,a}$;
 %\footnote{also follows from the fact that
 %$J^{\rm aff}\subseteq J$ is the greatest affine connected (automatically normal here) subgroup, so that it must contain
 %$\check{J}^{\rm aff}$}, arises from  the natural factorization of the natural inclusion
 %$\check{J}^{\rm aff} \to J$ that comes to being since the natural injective morphism
 %$\check{J}^{\rm aff}/ (\check{J}^{\rm aff}\cap J^{\rm aff}) \to J^{\rm ab}$, being one of a connected
%affine group-variety into an abelian variety, is necessarily constant; 
$w$ is the natural map induced 
 by the commutativity of the l.h.s. square.
 
 \n
 The snake lemma yields a natural exact sequence:
 \[
 0 \to \ke {\,u} \to \ke{\,v} \to \ke{\,w} \to \coke{\,u} \to \coke{\,v} \to \coke{\,w} \to 0,
 \]
 which, in view of the fact that $u,v$ are injective,  reduces to 
 \[
 0\to  \coke{\,u}/\ke{\, w} \to J_{n,a}/\check{J}_{n,a} \to  J^{\rm ab}_{n,a}/ (\check{J}^{\rm ab}_{n,a}/\ke{\, w}) \to 0.
 \]
 Since $\ke{\,w}$ sits inside the abelian variety $\check{J}^{\rm ab}$ and inside the affine $\coke{\,u}$, 
 it is a finite group.
 
 \n
 Since $\coke{\,u}/\ke{\, w}$ is affine, connected,  and  sits inside the abelian variety $J_{n,a}/\check{J}_{n,a}$,
 it is trivial. It follows that $\coke{\,u}=\ke{\, w}$, and since $\coke{\,u}$ is connected, so is the finite  $\ke{\,w}$ which is thus trivial. In particular, $\coke {\,u}$ is  also trivial and $\check{J}^{\rm aff}_{n,a} = J^{\rm aff}_{n,a}.$
 
 \n
 It follows that $J_{n,a}/\check{J}_{n,a} = J^{\rm ab}_{n,a}/\check{J}^{\rm ab}_{n,a}$,
and  we are done.
\end{proof}

\subsection{Product structures}\la{prd}$\;$

\begin{lm}\la{lm1}
There is the cartesian diagram  with  $q, q'$  isomorphisms
\beq\la{codgsq}
\xymatrix{
H^0(C,D)  \times  M_n(0)  \ar[d]_{{\rm Id} \times  h_n(0)} \ar[r]^{\hskip 12mm q'}_{\hskip 11mm \sim} & M_n \ar[d]^{h_n} \\
H^0(C,D) \times  A_n(0)   \ar[r]^{\hskip 11mm q}_{\hskip 11mm \sim} & A_n.
}
\eeq

\end{lm}
\begin{proof}
The map $q'$ is defined by the assignment 
$(\s, (E, \phi)) \mapsto (E, \phi + \s {\rm Id}_E)$.
Since  $\phi$ preserves a subsheaf of $E$ if and only if $\phi  + \s {\rm Id}_E$ does the same, we have    that
 $q'$ preserve stability.
%A priori,  $q'$ may target  an irreducible  component of $\ms{M}$ different from $M_n:=\ms{M}_o$
%(cf. \S\ref{aninnov}); however, this is not the case, since $q'$ does not effect the underlying vector bundle
%and $M_n(0)$ and $M_n$ are  precisely the  irreducible components containing  the respective  loci of %Higgs bundles with underlying stable bundles. 
The inverse assignment  to $q'$ is  $(E, \phi) \mapsto \left(\frac{{\rm tr}(\phi)}{n}, (E, \phi - \frac{{\rm tr}(\phi) }{n}{\rm Id}_E
)\right)$.

\n
Let $p(M)(t) = \det (t{\rm Id}-M)=\sum_{i=0}^n (-1)^i m_i t^{n-i}$ be the characteristic polynomial of an  $n\times n$ matrix $M$. Let $s$ be a scalar. Then a simple calculation shows that

%simple calculation below, hidden
\iffalse
\[
\sum_{x=0}^n (-1)^x m_x (t-s)^{n-x}= \sum_{x=0}^n (-1)^x m_x  \sum_{y=0}^{n-x}  (-1)^y{n-x\choose y} s^y t^{n-x-y}
\]
\[
=
 \sum_{x=0}^n  \sum_{y=0}^{n-x} (-1)^{x+y} {n-x\choose y} m_x     s^y t^{n-x-y}
\]
set $i:= x+y$, $j=y$, then one domain of summation $0\leq x \leq n$, $0\leq y \leq n-x$ goes to
$0\leq i \leq n$, $0\leq j \leq i$ and we get
\[
=  \sum_{i=0}^n  \sum_{j=0}^{i} (-1)^{i} {n-i+j\choose j} m_{i-j}     s^j t^{n-i}.
\]
\fi
\beq\la{newton}
p(M+s {\rm Id})(t) = \sum_{i=0}^n
(-1)^i \left[ m_i + \sum_{j=1}^{i-1} {n-i+j \choose j} m_{i-j} s^j  + {n\choose i} s^i \right] t^{n-i},
\eeq
where we have broken up the summation in square bracket to emphasize that the coefficients of $t^{n-i}$
is linear in $m_i,$ and to  identify the coefficient of  $s^i$. 

\n
The shape of $q$ is dictated by the desire  to have (\ref{codgsq})  commutative and by  the relation 
(\ref{newton}) between the characteristic polynomial
of $\phi$  and the one of  $ \phi+ \s {\rm Id}_E$.
We thus define $q$ by the   assignment  (N.B.: there is no $u_1$, so $j\neq i-1$, hence the upper bound $j=i-2$
in the summation below)
\beq\la{sub}
(\s, u_2, \ldots u_n)
\longmapsto \left(n\s, \left\{u_i + \sum_{j=1}^{i-2}  {n-i+j\choose j} \s^j u_{i-j} +{n\choose i} \s^i\right\}_{i=2}^n \right).
\eeq
E.g.: $q: (\s, u_2,u_3) \mapsto (3\s, u_2+3\s^2, u_3+u_2\s + \s^3).$
A simple recursion,  based on the fact that $u_i$ appears linearly in the component labelled by $i$, shows that the assignment above  can be inverted and that $q$ is an isomorphism.

\n
It is immediate to verify that the square diagram is commutative. 
%Since $q'$ is an isomorphism and
%${\rm Id} \times h_n(0)$ is proper, then so is $h_n\circ q$ and so is $q$. 
%Since $q$ is proper and bijective and $M_n$ is normal, we have that $q$ is an isomorphism.
Since  the morphisms  $q$ and $q'$ are isomorphisms,
the diagram is cartesian.
\end{proof}

\begin{lm}\la{lm2}
There is a natural commutative diagram  of proper morphisms with cartesian square
\beq\la{codgga}
\xymatrix{
{\rm Pic}^0(C)  \times \check{M}_n  \ar[dd]_{{\rm pr}_1} \ar[rr]^r  \ar[rd]_{\check{h}_n \circ {\rm pr}_2} && M_n(0)\ar[dd]^{\det}
\ar[ld]^{h_n(0)} \\
&  \check{A}_n=A(0) & \\
{\rm Pic}^0(C) \ar[rr]^{r'} & & {\rm Pic}^e(C)
}
\eeq
with $r$ and $r'$ proper Galois \'etale covers with Galois group the finite subgroup ${\rm Pic}^0[n] \subseteq
{\rm Pic}^0(C)$ of line bundles of order $n$. 
\end{lm}
\begin{proof}
The map $r$ is defined by the assignment $(L, (E, \phi))
\mapsto (E \otimes L, \phi \otimes {{\rm Id}_L}).$  Since $\check{M}_n$ is the closure of the loci of stable Higgs pairs  with stable underlying vector bundle, it is clear that, as indicated in (\ref{codgga}), $r$ maps into the closure $M_n(0)$ of the loci of stable Higgs pairs  with stable underlying vector bundle. 

\n
The map $r'$ is defined by the assignment $(L \mapsto \e  \otimes L^{\otimes n})$ (rem: $\e \in {\rm Pic}^e(C)$ is the fixed line bundle used to define $\check{M}_n$). The map $r'$
is finite, \'etale and Galois, with Galois group the subgroup ${\rm Pic}^0(C)[n] \subseteq {\rm Pic}(C)$ of
$n$-torsion points.

\n
By construction,  (\ref{codgga})   is  commutative. We need to show that the square is cartesian.
 
 \n 
Let $F$ be the fiber product of $r'$ and $\det.$  
Since $r'$ is \'etale and $M_n(0)$ is nonsingular, $F$ is nonsingular.
Since, by virtue of Lemma \ref{rt6},  $\det$ is smooth with integral fibers, 
then so  is  the natural projection $F \to {\rm Pic}^0(C),$ and $F$  is integral.
By the universal property of fibre products, we have a natural map $u:{\rm Pic}^0(C)  \times \check{M}_n \to F$ making the evident diagram commutative.
This map is bijective on closed points, where the inverse is given by
$(L, (E, \phi)) \mapsto (L, (E \otimes L^{-1}, \phi \otimes {\rm Id}_{L^{-1}}))$.
Since the domain and range of $u$ are nonsingular and $u$ is bijective, we conclude that
$u$ is an isomorphism: factor $u= f\circ j$, with $j$ an open immersion and $f$ finite and birational,
so that $f$ is necessarily an isomorphism, and $j$ is bijective, hence an isomorphism as well.
%(alternatively, one reaches the same conclusion via the modular interpretation of both terms).
\end{proof}

\subsection{Product structures, re-mixed}\la{lm1r}$\;$

In analogy with  Lemma \ref{lm1}, and keeping in mind the construction of spectral curves \S\ref{spcv}
as the  universal divisor inside of $V(D)\times A_n$,
we have the cartesian square diagram with $q,q''$ isomorphisms
\beq\la{c1}
\xymatrix{
H^0(C,D) \times A_n(0) \times  V(D) \ar[d]_{{\rm pr}_1 \times  {\rm pr}_2} \ar[r]^{\hskip 11mm q''}_{\hskip 11mm \sim} & A_n \times V(D) \ar[d]^{{\rm pr}_1} & (\s,u_\bullet, (x,v)) \mapsto (q(\s,u_\bullet), (x, v+\s))  \\
H^0(C,D) \times  A_n(0)   \ar[r]^{\hskip 11mm q}_{\hskip 11mm \sim} & A_n, &
}
\eeq
where $(\s, u_\bullet = (u_2, \ldots, u_n)) \in H^0(C,D)\times A(0)$ and $(x,v) \in V(D)_x$ is  the line fiber
of $V(D)$ over a point $x\in C$. For every  fixed $(\s, u_\bullet)$, the resulting morphism $q'': V(D)\stackrel{\sim}\to V(D)$
is, fiber-by-fiber, the translation  in the line direction by the amount $\s$ (linear change of coordinates
$t \mapsto t+\s)$ (cf. \ci[Rmk. 2.5]{ha-pa}).  

Consider the spectral curve family $\m{C} \subseteq A_n \times V(D)$ and the pre-image $\m{C}(0)$
of $A_n(0)$. 
Then, by restricting  $q''$ to $\m{C},$ we obtain  a cartesian square diagram
\beq\la{c2}
\xymatrix{
H^0(C,D)  \times  \m{C}(0)  \ar[d]_{{\rm Id} \times  p(0) } \ar[r]^{\hskip 11mm q'''}_{\hskip 11mm \sim} & \m{C} \ar[d]^p \\
H^0(C,D) \times  A(0)   \ar[r]^{\hskip 11mm q}_{\hskip 11mm \sim} & A,
}
\eeq
with  $q, q'''$  isomorphisms. For every fixed $(\s, u_\bullet) \in H^0(C,D) \times A_n(0)$, we have the spectral curve
$({\rm Id}\times p(0))^{-1} (\s, u_\bullet) = p(0)^{-1} (u_\bullet)=\m{C}_{0,u_\bullet}$
The morphism $q''$ maps $\m{C}_{0,u_\bullet}$ isomorphically onto $\m{C}_{q (\s, u_\bullet)}$,
via the fiber-by-fiber translation by the amount $\s$.

By recalling that $J_n(0)={J_n}_{|A_n(0)},$ and by setting $q^{iv}:= ((q''')^{-1})^*$, we obtain a cartesian square diagram with $q,q^{iv}$ isomorphisms
\beq\la{c3}
\xymatrix{
H^0(C,D)  \times  J_n(0)  \ar[d]_{{\rm Id} \times g_n(0) } \ar[r]^{\hskip 14mm q^{iv}}_{\hskip 13mm \sim} & J_n \ar[d]^{g_n} \\
H^0(C,D) \times  A_n(0)   \ar[r]^{\hskip 13mm q}_{\hskip 13mm \sim} & A_n.
}
\eeq

\subsection{$\delta$-regularity of  Prym over the elliptic locus}\la{9a}$\;$

Recall that  the elliptic locus $A_n^{\rm ell} \subseteq A_n$ is the  locus of characteristics $a \in A_n$
yielding geometrically  integral spectral curves $\m{C}_a.$  We denote by $A_n^{\rm ell} (0)$ and by $\check{A}^{\rm ell}_n$
the restriction of the elliptic locus to $A_n(0)=\check{A}_n.$
Recall   Definition \ref{delrg} ($\delta$-regularity).

\begin{pr}\la{1q2}
The group scheme $J_n(0)/A_n(0)$ is $\delta$-regular over $A_n^{\rm ell}(0)$.
The group scheme $\check{J}_n/\check{A}_n$ is $\delta$-regular over $\check{A}_n^{\rm ell}$, i.e.  if
 $a \in \check{A}_n^{\rm ell}$, then 
\beq\la{98bis}
d_a^{\rm ab}(\check{J}_n) \geq d_{\check{h}_n} - d_{\check{A}_n} + d_a \quad \mbox{\rm ($SL_n$ {\bf $\delta$-regularity inequality})}.
\eeq
\end{pr}
\begin{proof} 
Consider  the locally closed  ``strata" with invariant $d_a^{\rm aff}(-)= \delta$: 
\[
S_\delta:=S_\delta (J_n/A_n) \subseteq 
A_n, \quad 
S_\delta (0):=S_\delta (J_n(0)/A_n(0)) \subseteq 
A_n(0)  \quad
\check{S}_\delta:=S_\delta (\check{J}_n/\check{A}_n) \subseteq 
\check{A}_n.
\]
 By Proposition \ref{azz=}.(\ref{lk}), we have that $\check{S}_\delta =  S_\delta \cap A_n(0)= S_\delta (0)$.
 It follows that the two conclusions of the proposition are equivalent to each other, and  that it is enough to 
 prove the codimension assertion for $S_\delta(0)$.
 
 \n
By  Lemma \ref{gdr}, 
since we already know that ${\rm codim}_{A_n} (S_\delta) \geq \delta,$ for every $\delta \geq 0$ (see 
Lemma \ref{gdr}.(\ref{r4x}) and the Severi inequality in the
 proof of Theorem \ref{f4}),
 we need to make sure that intersecting with $A_n(0)$ does not spoil codimensions. This follows from (\ref{c3}),
 for it implies that
 \beq\la{efb}
 q^{-1} (S_\delta) = H^0(C,D) \times S_\delta (0),
 \eeq
 so that the codimensions of $S_\delta$ in $A_n=H^0(C,D) \times A_n(0)$, and  of
 $S_\delta (0)$ in $A_n(0),$ coincide.
\end{proof}

\subsection{The norm morphism $N_p^{\rm ab}$}\la{tnmq} $\;$

Fix a geometric point $a$ of $\check{A}_n.$
Recall the diagram (\ref{gh})  of finite morphisms of curves and let us focus on $\xi, p, \w{p}$.
We have the surjection  (\ref{q0}) $\xi^*:J_{n,a}= {\rm Pic}^0(\m{C}_a=:\sum_k m_k \Gamma_k) \to {\rm Pic}^0(\w{\m{C}_{a,red}}= \coprod_k \w{\Gamma_k})$. Keeping in mind the Chevalley devissage, we have the following commutative diagram of short exact sequences
completing the r.h.s. square in  (\ref{bdgr}) (recall that $C_a:= C\otimes k(a)$)
\beq\la{fgr}
\xymatrix{
0 \ar[r] &  \check{J}_{n,a}  \ar[r] \ar@{->>}[dd] & J_{n,a} \ar[r] \ar@{->>}[dd]^{\xi^*}  
\ar@/^4pc/[drrr]_{N_p} & J_{n,a}/\check{J}_{n,a} \ar[r] \ar@{->>}[dd]^= \ar[rrd] & 0 & \\
&&&&&  {\rm Pic}^0(C_a)  \ar@/^/[lllu]_{p^*} \ar@/_/[llld]^{\w{p}^*} \\
0 \ar[r] &  \check{J}^{\rm ab}_{n,a}  \ar[r]  & J^{\rm ab}_{n,a} \ar[r]    \ar@/_4pc/[urrr]^{N_p^{\rm ab}} & J^{\rm ab}_{n,a} /\check{J}^{\rm ab}_{n,a} \ar[r]
\ar[rru]  & 0, &
}
\eeq
where $N_p^{\rm ab}$ is the arrow induced by $N_p$, in view of the fact that,  since  $N_p$ 
has  target an abelian variety, it must be  trivial when restricted to  the connected and affine $\ke {\, \xi^*}=J_{n,a}^{\rm aff} \subseteq  J_{n,a}$.

The arrow $N_p^{\rm ab}$ is {\em not} the norm $N_{\w{p}}$ associated with the morphism $\w{p}$. In fact, we have the following lemma.

\begin{lm}\la{ha-pa} For every $L \in J_{n,a}$, we have 
\beq\la{eq01z}
N_p^{\rm ab} (\xi^* L) = \bigotimes_k N_{\w{p}_k} ( \xi_k^*  L)^{\otimes m_k}, 
\eeq
\beq\la{eq02z}
N_{\w{p}} (\xi^*L) =  \bigotimes_k N_{\w{p}_k} ( \xi_k^* L).
\eeq
\end{lm}
%\[
%N_p^{\rm ab} (\xi^* L) = \otimes_k N_{\w{p}_k} ( \xi^*  L)^{\otimes m_k} = \otimes_k N_{p''_k} ( \xi_{2,k}^* \xi_{1,k}^* L)^{\otimes m_k}
%= \otimes_k N_{p'_k} (\xi^*_{1,k} L),
%\]
\begin{proof}
Again, recall diagram (\ref{gh}).
We have the following chain of identities:
\[
N_p^{\rm ab} (\xi^* L) = N_p (L)  = \bigotimes_k N_{p'_k} (\xi^*_{1,k} L) 
= \bigotimes_k N_{p''_k} ( \xi_{2,k}^* \xi_{1,k}^* L)^{\otimes m_k}
= \bigotimes_k N_{\w{p}_k} ( \xi^*  L)^{\otimes m_k},
\]
where:  the first  identity is by the definition of $N_p^{\rm ab},$ for   $N_p$ has descended
via the surjective  $\xi^*: J_{n,a} \to J^{\rm ab}_{n,a}$, which has $\ke {\, \xi^*}= J^{\rm aff}_{n,a};$
the  second identity  follows from \ci[Lemma 3.5]{ha-pa}, applied to the morphisms $\xi_{1,k}$, by keeping in mind
that the norm from a disjoint union is the tensor product of the norms from the individual connected components;
the  third  identity  follows from \ci[Lemma 3.6]{ha-pa},  applied to the morphisms $\xi_{2,k}$;
the  fourth identity  follows from \ci[Lemma 3.4]{ha-pa},  applied to the morphisms $\xi_{3,k}$.
This proves (\ref{eq01z}). 

\n
The identity (\ref{eq02z}) can be proved in the same way (without recourse
to [loc.cit., Lemma 3.6].
\end{proof}
   
\subsection{The Tate module of Prym  is polarizable}\la{itpz}$\;$
   
\begin{lm}\la{lzt}$\;$
Let $a$ be any geometric point of  $\check{A}_n$.   Let $\w{p}_a$, etc. be the corresponding morphisms
in {\rm  (\ref{gh})}.
   \ben
   \item
    $T_{\oql}(\w{p}_a^*)$ and $T_{\oql}(N_{p_a}^{\rm ab})$ are adjoint w.r.t. the bilinear forms $TW^{\rm ab}_a$ and $TW_{C_a}$.

    \item
    $\ke  \,(T_{\oql}(N_{p_a}^{\rm ab}) ) = T_{\oql} (\check{J}_a^{\rm ab})$.
        
    \item
    $N_{p_a}^{\rm ab} \circ \w{p}_a^*=  n \, {\rm Id}_{{\rm Pic}^0(C_a)}.$
    
   \een 
    
    \end{lm}
    
    \begin{proof}
    Recall that we have the spectral cover $\m{C}_a =\sum_k m_k \Gamma_k \to C_a=C\otimes k(a).$
    We start with (1).
    For every $\w{\gamma} = \sum_k \w{\gamma}_k \in T_{\oql} (J^{\rm ab}_{n,a})=
    \oplus_k T_{\oql} ( {\rm Pic}^0(\w{\Gamma}_k))$, and for every $c \in T_{\oql} ({\rm Pic}^0(C_a))$, we have  that
    \[
    TW^{\rm ab} \left(\w{\gamma}, T_{\oql}(\w{p}^*) (c)\right)=
    TW^{\rm ab} \left(\sum_k \w{\gamma}_k, \sum_k  T_{\oql}(\w{p}_k^*) (c) \right)
    = \sum_k m_k\,TW_{\w{\Gamma}_k} \left(  \w{\gamma}_k, T_{\oql}(\w{p}^*_k) (c)  \right)=
    \]
    \[
    = \sum_k  m_k\, TW_C \left(  T_{\oql} (N_{\w{p}_k}) (\w{\gamma_k}), c  \right) 
    = TW_C \left(   \sum_k  m_k\, T_{\oql} (N_{\w{p}_k}) (\w{\gamma_k}), c  \right) =
    \]
    \[
    = TW_C \left(T_{\oql} (N_p^{\rm ab}) (\w{\gamma}), c\right),
    \]
    where: the first identity is simply by consideration of components; the second identity follows form the fact that
    $TW^{\rm ab}$ is obtained from $TW,$ which is the direct sum of the individual $TW_{\w{p}_k}$, weighted by $m_k$ (see the end of the proof
    of Proposition \ref{poltmd}); the third identity is the classical  adjunction relation  
    (cf. \ci[p.186, equation I]{Mu} and \ci[Cor. 11.4.2, especially p.331, equation (2)]{Bi-La}) between norm and pull-back
    for the morphism $\w{p}_k: \w{\Gamma}_k \to C_a;$ the last equality is obtained by applying
    the functor $T_{\oql}$ to the identity (\ref{eq01z}), and (1) is proved.
    
    \n
  We prove (2).  The lower line in (\ref{fgr}) yields, in view of the isogeny (\ref{hg}), the short exact sequence
    \[
    \xymatrix{
    0 \ar[r]  &T_{\oql} (\check{J}^{\rm ab}_{n,a})  \ar[r] & T_{\oql} (J^{\rm ab}_{n,a})
    \ar[r] & T_{\oql} (J^{\rm ab}_{n,a}/(\check{J}^{\rm ab}_{n,a})
    \cong T_{\oql}({\rm Pic}^0(C_a)) \ar[r] & 0,
    }
    \]
     so that the resulting arrow  $T_{\oql} (J^{\rm ab}_{n,a}) \to T_{\oql} (J^{\rm ab}_{n,a}/\check{J}^{\rm ab}_{n,a})$ 
     gets identified with 
     \[
   %  \ke\, \left\{ T_{\oql} 
     %(N_p^{\rm ab})  :  
    % T_{\oql} ( J^{\rm ab}_{n,a}) \lorw T_{\oql}({\rm Pic}^0(C) ) 
     %\right\}.
   T_{\oql} 
     (N_p^{\rm ab})  :  
     T_{\oql} ( J^{\rm ab}_{n,a}) \lorw T_{\oql}({\rm Pic}^0(C_a) ).
     \]
     We prove (3).
    Recall the standard identity $N_{\w{p}_k}\circ \w{p}_k^*= n_k {\rm Id},$ 
    %(e.g., see \ci[ \S3.1.(8)]{ha-pa}, 
    and that $n = \sum_k n_k m_k.$
    Then (3) follows from Lemma \ref{ha-pa}:  for every $L \in   {\rm Pic}^0(C_a)$, we have
    \[
    N_p^{\rm ab} ( \w{p}^* L)= N_p^{\rm ab} ( \xi^* p^* L)
    = \bigotimes_k N_{\w{p}_k} (\xi_k^* p_k^* L)^{\otimes m_k}= 
    \]
    \[
    =  \bigotimes_k N_{\w{p}_k} (\w{p}_k^* L)^{\otimes m_k}=
    \bigotimes_k L^{\otimes n_k m_k} = L^{\otimes n}.
    \]
\end{proof}

\begin{tm}\la{slnpol} {\rm ({\bf Polarizability of the  Tate module of Prym})}
The restriction \[\check{TW}: T_{\oql}(\check{J}_n) \otimes T_{\oql}(\check{J}_n) \to \oql (1)\] of the Tate-Weil pairing $TW: T_{\oql}(J_n) \otimes T_{\oql}(J_n) \to \oql (1)$
is a polarization of the Tate module $T_{\oql}(\check{J}_n)$ on $\check{A}_n.$
\end{tm}
\begin{proof}
We fix an arbitrary geometric point $a$ of  $ \check{A}_n.$ By Proposition \ref{azz=}, we have that $\check{J}^{\rm aff}_{n,a} = 
J^{\rm aff}_{n,a}.$ We have already verified that $TW$ is trivial on the ``affine part"
$J_{n,a}^{\rm aff}= \check{J}_{n,a}^{\rm aff}$  (see the proof of  Proposition
\ref{poltmd} and (\ref{lk})).
It follows that $TW$ is trivial on $T_{\oql} (\check{J}_{n,a}^{\rm aff})$. We need to show that
the descended non-degenerate $TW^{\rm ab}$  (cf. Theorem \ref{poltmd}) on $T_{\oql} (J_{n,a}^{\rm ab})$ stays 
non-degenerate on $T_{\oql} (\check{J}_{n,a}^{\rm ab})$.

\n
By Lemma \ref{lzt}.(3), we have that 
\[
T_{\oql} (J^{\rm ab}_{n,a}) =  \ke{\,( T_{\oql}(N_p^{\rm ab}))} \oplus \im{\, (\w{p}^*)}.
\]
By Lemma \ref{lzt}.(1), the direct sum decomposition is orthogonal w.r.t. $TW^{\rm ab}$.

\n
By Lemma \ref{lzt}.(2), we may re-write the  orthogonal direct sum decomposition above as follows 
\[
T_{\oql} (J^{\rm ab}_{n,a}) =  T_{\oql} (\check{J}^{\rm ab}_{n,a}) \oplus^{\perp_{TW^{\rm ab}}}
\im{\, (\w{p}^*)},\]
so that the non-degenerate form  $TW^{\rm ab}$ restricts to a non-degenerate form on $T_{\oql} (\check{J}^{\rm ab}_{n,a})$.
\end{proof}

\subsection{Recap for the   $SL_n$ weak abelian fibration}\la{33}$\;$

Theorem \ref{f4} tells us that in the $D$-twisted, $GL_n$ case, the triple $(M_n, A_n,J_n)$
is a weak abelian fibration that is $\delta$-regular over the elliptic locus.   

Proposition \ref{1q2},
implies that the analogous conclusion holds  for $(M_n(0), A_n(0), J_n(0)$. In  fact, the polarizability of the Tate module
is automatic when restricting from $A_n$ to $A_n(0)$: for $J_n(0) = {J_n}_{|A_n(0)}$, and the Tate module is the restriction of the Tate module. Similarly, the stabilizers are the same and they are thus affine. Even though the Chevalley devissages
are un-effected when  passing from $A_n$ to $A_n(0)$, it is not a priori evident that
the $\delta$-regularity should be preserved (intersecting may spoil codimensions), and this is precisely  what
Proposition
\ref{1q2} ensures. 

The $D$-twisted  $SL_n$ case, i.e. $(\check{M}_n, \check{A}_n, \check{J}_n)$,   is slightly trickier because, in addition
to the discussion in the previous paragraph,
 the polarizability Theorem \ref{slnpol}
for  the Tate module $T_{\oql}(\check{J}_n)$  did not follow immediately from the $GL_n$ analogous
Theorem \ref{poltmd}.

We record for future use the  following result.

\begin{tm}\la{95}
The triple  $(\check{M}_n, \check{A}_n,\check{J}_n)$ is a weak abelian fibration  which is  $\delta$-regular over 
$\check{A}_n^{\rm ell}$. In particular,
\ben
\item
If $a\in {\socle}({R{\check{h}_n}}_* {\oql})$, then 
\beq\la{99biz}
d_{\check{h}_n} - d_{\check{A}_n} + d_{\check{a}} \geq d_{\check{a}}^{\rm ab}(\check{J}_n)  \qquad \mbox{\rm ({\bf Ng\^o support inequality})}.
\eeq
\item
If $a\in \check{A}_n^{\rm ell}$, then 
\beq\la{98biz}
d_{\check a}^{\rm ab}(\check{J}_n) \geq d_{\check{h}_n} - d_{\check{A}_n} + d_{\check{a}} \qquad \mbox{\rm ({\bf $\delta$-regularity inequality})}.
\eeq
\een
\end{tm}
\begin{proof}
The projective morphism $\check{h}_n: \check{M}_n \to \check{A}_n$ is of pure relative dimension 
$d_{\check{h}_n} = d_{h_n} -g$ (Prop. \ref{kh}).  By (\ref{oim}), the pure relative dimension 
$d_{g_n} =d_{h_n}$. By the very  construction \S\ref{jnc} of $\check{J}_n$,  the pure relative dimension 
of $\check{g}_n: \check{J}_n \to \check{A}_n$ is $d_{\check{g}_n}= d_{g_n}-g.$ It follows that $d_{\check{h}_n}=
d_{\check{g}_n}.$  The stabilizers of the $\check{J}_n$-action are affine because they are  closed subgroups
of the stabilizers of the $J_n$-action, which are affine by virtue of  Proposition \ref{affstbz}.
The Tate module $T_{\oql} (\check{J}_n)$ is polarizable by virtue of Theorem \ref{slnpol}.
We have thus verified that the triple is a weak abelian fibration. In particular, Ng\^o support inequality 
\ref{nst} implies (\ref{99biz}). The $\delta$-regularity assertion is contained in Proposition \ref{1q2}.
\end{proof}

%\subsection{Polarizability of the Tate module $T_{\oql} (\check{J}_n)$ on $\check{A}_n$}\la{poltm}$\;$
%\ci{bi-la}, \S11.4, equation (2).  \ci{muabvar}, Ch.IV, \S20, eqaution (I) and the Definition directly below it.

\subsection{Endoscopy and the $SL_n$ socle  over the elliptic locus}\la{zel}$\;$

We employ the notation and results  in \S\ref{vbn}, especially Fact \ref{num}.

According to \ci[Prop. 6.5.1]{ngofl}, we have 
\beq\la{rb6}
\left({R^{2\check{h}_{n}}{{\check h}_n}}_* \oql \right)_{| \check{A}_n^{\rm ell}} \cong
{ \oql}_{\check{A}_n^{\rm ell}} \bigoplus \bigoplus_\Gamma 
{ \oql}_{\check{A}_{n,\Gamma}^{\rm ell}}^{\oplus o_\Gamma -1} \qquad
(\Gamma, \, o_\Gamma  \mbox{ as in  (\ref{e45})}).
\eeq

In view of Theorem \ref{95}, the triple
 $(\check{M}_n, \check{A}_n, \check{J}_n)$ is a weak abelian fibration 
that is $\delta$-regular over $\check{A}_n^{\rm ell}$, so that we may use 
Ng\^o support theorem \ci[Thm. 7.2.1]{ngofl}, to the effect that the supports over the elliptic
locus must also be the supports appearing in (\ref{rb6}), and conclude that
\beq\la{mopz}
{{\socle}} ({R{{\check h}_n}}_* \oql) \cap \check{A}_n^{\rm ell} \, = \, \left\{\eta_{\check{A}_n}\right\} \coprod  \coprod_\Gamma \{ \eta_{\check{A}_{n,\Gamma}}\}.
\eeq

\section{Multi-variable weak abelian fibrations}\la{45}$\;$

While the $SL_n$ support inequality is used in the proof of our main Theorem \ref{maintm} on the $SL_n$ socle, the
$SL_n$ $\delta$-regularity inequality is of no use in that respect.  \S\ref{45} is devoted to 
establish the $\delta$-regularity-type inequality that we need instead, i.e. (\ref{speremb}).
To this end,  \S\ref{qwoz} introduces the   multi-variable $GL_n$ weak abelian fibration 
$(M_{n_\bullet},  A_{n_\bullet}, J_{n_\bullet})$.  \S\ref{qw2}  introduces 
its $m_\bullet$-weighted-traceless counterpart $(M_{n_\bullet m_\bullet}(0),  A_{n_\bullet m_\bullet}(0), 
J_{n_\bullet}(0))$, and   establishes a series  of 
product-decomposition-formulae of the form $H^0(C,D) \times (-)_{n_\bullet m_\bullet} (0)\cong 
(-)_{n_\bullet}$. This construction yields the group scheme
$J_{n_\bullet m_\bullet} (0)/A_{n_\bullet m_\bullet} (0)$ with the useful $\delta$-regularity-type inequality that we need.
Extracting it, as it is done in \S\ref{rt5}, is not a priori completely evident:  one has trivially a $\delta$-regularity-type
inequality
for the multi-variable Jacobi groups scheme $J_{n_\bullet m_\bullet}/A_{n_\bullet m_\bullet}$,  which takes the form of an inequality
for codimensions of $\delta$-loci  in $A_{n_\bullet m_\bullet}$; however,  one needs instead  to control  the codimensions 
of the $\delta$-loci {\em  after} restriction to  the linear subspace $A_{n_\bullet m_\bullet}(0),$ which is not meeting
the $\delta$-loci transversally.

\subsection{The weak abelian fibration $(M_{n_\bullet}, A_{n_\bullet}, J_{n_\bullet})$}\la{qwoz}
$\;$

Let $n_{\bullet}=(n_1, \ldots,  n_s)$ be a finite sequence of positive integers.
Define 
\beq\la{weg}
(M_{n_\bullet}, A_{n_\bullet}, J_{n_\bullet}) := \left(\prod_k M_{n_k}, \prod_k A_{n_k}, \prod_k J_{n_k}\right)
\eeq
\beq\la{weg1}
A_{n_\bullet}^{\rm ell}:= \prod_k A_{n_k}^{\rm ell}.
\eeq
A geometric point of $A_{n_\bullet}^{\rm ell}$ correspond to an ordered $s$-tuple of geometrically integral
spectral curves $(\Gamma_1, \ldots, \Gamma_s)$ of  respective spectral degrees $(n_1, \ldots, n_s)$.

The requirements of Definition \ref{waf} (same pure relative dimensions, affine stabilizers,
polarizability of Tate modules, $\delta$-regularity on the elliptic locus) are met on each factor
separately by virtue of Theorem \ref{f4}. (In verifying $\delta$-regularity, one needs a simple application
of Lemma \ref{gdr}.(2) to each factor:
let $a\in A_{n_\bullet}^{\rm ell}$; let $x_\bullet$ be a closed general point in $\ov{\{a\}}$; let $a_k$ be the projection of $a$ to the $k$-th factor; then: $a_k \in A_{n_k}^{\rm ell}$, $x_k$ is a closed general point of $\ov{\{a_k\}}$, and $\sum_k d_{a_k} \geq d_a$ (because
$\ov{\{a\}} \subseteq \prod_k \ov{\{a_k\}}$); we have $d_a^{\rm ab}(J_{n_\bullet})=d_{x_\bullet}^{\rm ab}(J_{n_\bullet})=
\sum_k d_{x_k}^{\rm ab}(J_{n_k})=\sum_k d_{a_k}^{\rm ab}(J_{n_k}) \geq \sum_k (d_{a_k}(J_{n_k}) - d_{A_{n_k}}+ d_{a_k})
= \sum_k d_{a_k}(J_{n_k})  - d_{A_{n_\bullet}}+ \sum_k d_{a_k} \geq   \sum_k d_{a_k}(J_{n_k})  - d_{A_{n_\bullet}}+ d_{a}=
\sum_k d_{x_k}(J_{n_k})  - d_{A_{n_\bullet}}+ d_{a}= d_{a}(J_{n_\bullet})  - d_{A_{n_\bullet}}+ d_{a}$.)
 It follows immediately that they are met on the product, so that
 (\ref{weg}) is a
weak abelian fibration which is $\delta$-regular over $A_{n_\bullet}^{\rm ell}.$

\subsection{Stratification by type of the $GL_n$ Hitchin base $A_n$}\la{stry}$\;$

Let $n \in \zed^{\geq 1}$ and let $s \in \zed^{\geq 1}$ with $ 1 \leq s \leq n.$ We consider the set  $NM (s)$ of pairs
$(n_\bullet, m_\bullet)$ subject to the following requirements:
1) $n_1 \geq \ldots \geq n_s$; 2)  $m_k \geq m_{k+1}$ whenever $n_k=n_{k+1}$; 3) $\sum_{k=1}^s m_k n_k =n$.
There is the  partition  of the integral variety
\beq\la{part}
A_n = \coprod_{1\leq s \leq n} \,\coprod_{(n_\bullet, m_\bullet) \in MN(s)} S_{n_\bullet m_\bullet}
\eeq
into the locally closed integral subvarieties 
\beq\la{str}
S_{n_\bullet m_\bullet}:= \left\{a \in A_n \, |\, \cu{C}{\ov{a}} = \sum_{k=1}^s m_k \cu{C}{k,\ov{a}}\right\}
\subseteq A_n,
\eeq 
where
$\ov{a} \to a$ is given by an algebraic closure $k(a) \subseteq \ov{k(a)}$, and 
each spectral curve $\cu{C}{k,\ov{a}}$  is irreducible of spectral curve degree $n_k$. The closure 
 $\ov{S_{n_\bullet, m_\bullet}} \subseteq A_n$ is the image of the   finite morphism
 (\ci[\S9]{ch-la}) 
\beq\la{lam}
\lambda_{m_\bullet n_{\bullet}}: A_{n_\bullet} \to A_n,   \qquad \im{\, ( \lambda_{n_\bullet m_\bullet})} = 
\ov{S_{n_\bullet m_\bullet}} \subseteq A_n,
\eeq
  which on closed points is defined as follows:
 $(a_1, \ldots , a_s) \mapsto a$, where we view $a_k$ as a characteristic polynomial $P_{a_k}$ of degree $n_k$, we consider
 the degree $n$ polynomial $\prod_{k=1}^s P^{m_k}_{a_k}$,  and we  take $a$ to be the corresponding closed point on $A_n.$
 The stratum $S_{n_\bullet m_\bullet}$ is the image of a suitable Zariski dense open subvariety inside the Zariski dense
 open subvariety $\prod_{k=1}^s A_{n_k}^{\rm ell} \subseteq  A_{n_\bullet}$. Given a point $a \in A_n$, we have $a \in S_{n_\bullet m_\bullet}$
 for a unique triple $(s, (n_\bullet, m_\bullet))$, with  $1\leq s\leq n$ the number of irreducible components of $\cu{C}{\ov{a}}$, 
and with  $(n_\bullet, m_\bullet)  \in  NM(s) $, which we call the type of $a \in A_n$.   Since the spectral curve $\cu{C}{a}$ may have a strictly smaller number of components than $\cu{C}{\ov a}$, the type  of $a$ is observed on $\cu{C}{\overline{a}}$.  

Geometrically, we may think of the morphism $\lambda_{n_\bullet m_\bullet}$ as sending an ordered $s$-tuple
of integral curves $(\Gamma_1, \ldots, \Gamma_s)$,  to the spectral curve denoted (\S\ref{spcv}) by $\sum_k m_k \Gamma_k$. As it is already clear in the case $s=2$,
with 
$(n_1,n_2; m_1,m_2) = (1,1;1,1),$ in general, the finite morphisms $\lambda_{n_\bullet m_\bullet}$ are not birational.

The morphisms (\ref{lam}) are introduced in  \ci[\S9]{ch-la} in order to exploit the  $GL_n$ $\delta$-regularity
inequalities for each $J_{n_k}^{\rm ell}/A_{n_k}^{\rm ell}$, $k=1,\ldots, s$ (however, see  Remark \ref{nt}).

The resulting inequalities are of no use to us for the $SL_n$ case: they are too weak. One may be tempted to replace them by taking the multi-variable counterpart to 
the $SL_n$ $\delta$-regularity inequality  (\ref{98biz}). As it turns out, these  $SL_n$ inequalities are also of no use to us towards the proof of Theorem \ref{maintm} on the  $SL_n$ socle:
they not relevant in the proof  given in \S\ref{pmaintm} of Theorem \ref{maintm}
(the $SL_n$ support inequality (\ref{99biz}) plays a crucial role, though).  

The  multi-variable $\delta$-regularity inequalities that we need
for the proof of Theorem \ref{maintm}  on the $SL_n$ socle are given by Corollary \ref{lm5b}.(\ref{speremb}), and are
to be extracted from the constructions of the next  \S\ref{qw2}.

\subsection{The weak abelian fibration  $(M_{n_\bullet m_\bullet } (0), A_{n_\bullet m_\bullet}(0) , J_{n_\bullet m_\bullet}(0)$}\la{qw2}$\;$

Define  what we  may call  the subspace of  multi-weighted-traceless
characteristics by setting (recall that $a(1)$ is the trace-component of a characteristic)
\beq\la{dr}
A_{n_\bullet m_\bullet} (0):= \{ (a_1, \ldots, a_s)\,|\; \sum_k m_k a_k(1)=0\}
\subseteq A_{n_\bullet}.
\eeq
This is a vector subspace of codimension $h^0(C,D)=d-(g-1)$. Define 
$M_{n_\bullet m_\bullet} (0):= h_{n_\bullet}^{-1}(A_{n_\bullet m_\bullet} (0)) \subseteq M_{n_\bullet}$
(given its reduced structure; we are about to verify the statement associated with (\ref{c4}), so that,  a posteriori,
this pre-image is indeed  automatically reduced).

What follows is in direct analogy with the constructions in the proof of  Lemma \ref{lm1}, and in its re-mixed
version in  \S\ref{lm1r}.
 We have the cartesian square  diagram 
 \beq\la{c4}
 \xymatrix{
H^0(C,D)  \times  M_{n_\bullet m_\bullet} (0)  \ar[d]_{{\rm Id} \times h_{\n_\bullet}(0) } \ar[r]^{\hskip 14mm \frak q'}_{\hskip 14mm \sim} & M_{n_\bullet} \ar[d]^{h_{\bullet}} \\
H^0(C,D) \times  A_{n_\bullet m_\bullet} (0)   \ar[r]^{\hskip 14mm \frak q}_{\hskip 14mm \sim} & A_{n_\bullet},
}
 \eeq
with ${\frak q},{\frak q'}$ isomorphisms,
where:

\ben
\item
in analogy with (\ref{sub}), and by keeping in mind that here the entries $u_{k1}$ are not necessarily zero,
the map ${\frak q}$ is given by the assignment  sending
\[
\left(
\s, \left( u_{11}, \ldots, u_{1n_1}   \right),
\ldots
\left(  u_{s1}, \ldots, u_{sn_s}   \right)
\right),
\qquad
\mbox{subject to
$\sum_k m_k u_{k1}=0$,}
\]
to (having set $u_{k0}:=1$, for convenience)
\[
\left(
\left\{\sum_{j=0}^i {n_1-i+j\choose j} \s^j u_{1,i-j}\right\}_{i=1}^{n_1}, \ldots , 
\left\{\sum_{j=0}^i {n_s-i+j\choose j} \s^j u_{s,i-j} \right\}_{i=1}^{n_s}
\right);
\]
\item
$\frak q'$ is defined by the assignment
\[
\left( \s, \left( E_1,\phi_1 \right), \ldots, \left(E_s, \phi_s \right) \right)
\longmapsto
\left( \left( E_1,\phi_1+\s {\rm Id} \right), \ldots, \left(E_s, \phi_s  +\s {\rm Id} \right) \right).
\]
\een
As in the proof Lemma \ref{lm1}, a simple recursion yields the map inverse to $\frak q$, whereas the one inverse to 
$\frak q'$ is given by the assignment (rem: $n = \sum_{k} m_k n_k$)
%{\cm how I got below: use $s=2$; send $\{\psi_k\}_{k=1}^s$ to $(?, \{?_k\}_{k=1}^s)$; guesswork: write $?_k=
%\psi_k +$ (linear combination of all the ${\rm tr} (\psi_j)$) etc. \ldots}
 \[
\left\{ (E_k, \psi_k)  \right\}_{k=1}^s      
\longmapsto 
\left(  \frac{\sum_j  m_j {\rm tr}(\psi_j)}{n},  
\left\{ \left(E_k, \psi_k -\sum_j \frac{m_j}{n}{{\rm tr}(\psi_j)}   {\rm Id}\right) 
\right\}_{k=1}^s
\right)
.
\]

Finally, by setting $J_{n_\bullet m_\bullet} (0):= {J_{n_\bullet}}_{| A_{n_\bullet m_\bullet}(0)}$, we have the cartesian
square diagram with $\frak{q}''$ and $\frak{q}$ isomorphisms
\beq\la{e3}
\xymatrix{
H^0(C,D) \times J_{n_\bullet m_\bullet} (0) \ar[rr]^{\hskip 16mm \frak q''} \ar[d]_{{\rm Id} \times g_{n_\bullet m_\bullet}(0)} && J_{n_\bullet} \ar[d]^{g_{n_\bullet m_\bullet}} \\
H^0(C,D) \times A_{n_\bullet m_\bullet}(0) \ar[rr]^{\hskip 16mm \frak q}  && A_{n_\bullet},
}
\eeq
obtained in the same way as (\ref{c3}).

\subsection{Multi-variable $\delta$-regularity over the elliptic loci}\la{rt5}$\;$

Recalling  the definition of $S_\delta(J/A)$ in (\ref{bt}), we have the following identification of $\delta$-loci
 
\begin{pr}\la{l5}$\;$
${\frak q}^{-1} (S_\delta (J_{n_\bullet}/A_{n_\bullet}) ) = H^0(C,D) \times 
S_\delta (J_{n_\bullet m_\bullet}(0)/A_{n_\bullet m_\bullet}(0)).$
\end{pr}
\begin{proof}
Keeping in mind that $S_\delta (J_{n_\bullet}/A_{n_\bullet})$ is naturally stratified by
products of individual $S_{\delta_k}(J_{n_k}/A_{n_k})$ with $\sum_k \delta_k = \delta$,
the proof runs parallel to the one of Proposition \ref{1q2}, with (\ref{e3}) playing the role of (\ref{c3}).
\end{proof}

\begin{tm}\la{lm5}$\;$
The   weak abelian fibrations
\[
(M_n, A_n, J_n), \quad (M_n(0), A_n(0), J_n(0)), \quad (\check{M}_n, \check{A}_n, \check{J}_n),
\]
\[ 
(M_{n_\bullet}, A_{n_\bullet}, J_{n_\bullet}), \quad 
(M_{n_\bullet m_\bullet} (0), A_{n_\bullet m_\bullet}(0), J_{n_\bullet m_\bullet} (0))
\]
are $\delta$-regular  when restricted to their  respective elliptic loci
\[
A_n^{\rm ell}, \quad A_n^{\rm ell}(0):= A_n (0) \cap A_n^{\rm ell}, 
\quad  \check{A}_n^{\rm ell}:= \check{A}_n \cap A_n^{\rm ell}, 
\]
\[ 
A_{n_\bullet}^{\rm ell} :=   \prod_k A_{n_k}^{\rm ell}, \quad 
A_{n_\bullet m_\bullet}^{\rm ell}(0) :=  A_{n_\bullet m_\bullet}(0) \cap \prod_k A_{n_k}^{\rm ell}.
\] 
\end{tm}
  \begin{proof} We have already proved all the conclusions in the  single-variable case: we have displayed them for
  emphasis only. We have already observed in \S\ref{qwoz} that the single-variable case implies
  that  $(M_{n_\bullet}, A_{n_\bullet}, J_{n_\bullet})$
  is a weak abelian fibration  which is $\delta$-regular over its elliptic locus $A_{n_\bullet}^{\rm ell}$.
  
  \n
  By virtue of (\ref{c4}) and of (\ref{e3}), we see that 
  $(M_{n_\bullet m_\bullet} (0), A_{n_\bullet m_\bullet}(0), J_{n_\bullet m_\bullet} (0))$ is a weak abelian fibration
  as well, which, by virtue of Proposition \ref{l5}, is $\delta$-regular over its elliptic locus
  $A_{n_\bullet m_\bullet}^{\rm ell}(0)$ (cf. the proof of Proposition \ref{1q2}).
  \end{proof}

 \begin{rmk}\la{nt}
 The following claim \underline{does not} hold:
given  a point $a \in S_{n_\bullet m_\bullet}  \subseteq A_n$,  we can write   $a= \lambda_{n_\bullet m_\bullet} (a_1, \ldots a_s)$
for a suitable $s$-tuple $a_k \in A_{n_k}.$ This is true if $a$ is a closed point, but it fails in general.
This claim has been used in \ci[\S9, proof of main theorem]{ch-la}.
\end{rmk}

Corollary \ref{lm5b}, equation (\ref{sperema}) below  remedies    the minor inaccuracy in the proof of 
\ci[\S9, proof of main theorem]{ch-la} pointed out in Remark \ref{nt}. It also establishes
 the  $SL_n$-variant (\ref{speremb}) that we need in the course of the proof of Theorem \ref{maintm} in \S\ref{pmaintm}.

   \begin{cor}\la{lm5b} {\rm ({\bf Multi-variable $\delta$-inequalities})}
   Let $a \in A_n$ and let $(n_\bullet, m_\bullet) \in NM(s)$ be its type {\rm (\S\ref{qw2})}. 
   Then we have the following multi-variable $GL_n$ $\delta$-inequality
   \beq\la{sperema}
   d_a^{\rm ab} (J_n)
 \geq \sum_k \left(d_{h_{n_k}} - d_{A_{n_k}} \right)  
+ d_a.
\eeq
If, in addition, $a \in A_n(0)=\check{A}_n$,  then we have the following multi-variable $SL_n$ $\delta$-inequality
   \beq\la{speremb}
   d_a^{\rm ab} (J_n)
 \geq \sum_k \left(d_{h_{n_k}} - d_{A_{n_k}} \right)
+ [d-(g-1)] + d_a.
\eeq
%\beq\la{speremb}
 %  d_a^{\rm ab} (J_n)= d_{\alpha}^{\rm ab}(J_{n_\bullet})=
%\sum_k d_{a_k}^{\rm ab}(J_{n_k}) 
% \geq \sum_k \left(d_{h_{n_k}} - d_{A_{n_k}} \right)
%+ [d-(g-1)] + d_a.
%\eeq
   \end{cor}
\begin{proof}
Let $a \in A_n$ and let $V(a):=\ov{\{a\}} \subseteq A_n$ be the associated integral closed subvariety.
Let $(n_\bullet, m_\bullet) \in NM(s)$ be the type of $a$. 
 Let $\alpha$ be any point in the non empty  fiber  $\lambda^{-1}_{n_\bullet m_\bullet} (a)$.
Then $d_a:=\dim (\ov{\{a\}}) = \dim (\ov{\{\alpha\}})=
d_\alpha$.

\n We choose an algebraic closure of $k(a)$ that contains the 
finite field extension $k(a) \subseteq k(\alpha)$.  
We can identify the curves $\cu{C}{\ov{\alpha}}= \cu{C}{\ov{a}, {\rm red}}$, so that the two curves have the same number $s$ of geometrically irreducible components. It follows  that $\alpha \in 
A_{n_\bullet}^{\rm ell}$. By virtue of (\ref{3030}), it also follows that $d_a^{\rm ab} (J_n) = d_{\alpha}^{\rm ab}(J_{n_\bullet})$.

\n
The $\delta$-regularity inequality for $J_{n_\bullet}$ over $A_{n_\bullet}^{\rm ell}$ implies that
$d_{\alpha}^{\rm ab}(J_{n_\bullet}) \geq d_{h_{n_\bullet}} - d_{A_{n_{\bullet}}} + d_\alpha$, and (\ref{sperema}) follows.

\n
%there are more details on this in the Janaury_29_2016 version
Since $a$ has type $(n_\bullet, m_\bullet)$, we have that
$\alpha$ satisfies the weighted trace constraint (\ref{dr})  that defines  $A_{n_\bullet m_\bullet}^{\rm ell}$, so that
$\alpha \in A_{n_\bullet m_\bullet}^{\rm ell} (0)$.
Then  (\ref{speremb}) is proved in the same way as (\ref{sperema}) 
by using the $\delta$-regularity of $J_{n_\bullet m_\bullet} (0)$   over $A_{n_\bullet m_\bullet}^{\rm ell} (0)$,
and the facts that $d_{h_{n_\bullet}} = \sum_k d_{h_{n_k}}$,   and (cf. (\ref{dr})) $d_{A_{n_\bullet m_\bullet}}(0)  =
 \dim (A_{n_\bullet}) - h^0(C,D)= \sum_k d_{A_{n_k}} - [d- (g-1)]$.
\end{proof}

\section{Proof of the main Theorem \ref{maintm} on the $SL_n$ socle}\la{pfqa}$\;$

This \S\ref{pfqa} is devoted to the proof of our main Theorem \ref{maintm} on the $SL_n$ socle. 
\S\ref{si} collects some formulae.
\S\ref{pmaintm} contains the  proof of Theorem \ref{maintm}.

\subsection{A list of dimension formulae}\la{si}$\;$

We first list some dimensional formulae in the $GL_n$ case.
We set  $d_{M_n}:= \dim{M_n}$, $d_{A_n}:= \dim{A_n}$,  and $d_{h_n}:=
d_{M_n}- d_{A_n}.$ The dimension of $M_n$ is given by \ci[Prop. 7.1]{nitsure}; the dimension of 
$A_n=\oplus_{i=1}^n h^0(C, iD)$
is computed via Riemann-Roch;  the relative dimension $d_{h_n}$ is given by (\ref{rr}). We thus have
 \[
 d_{M_n} = n^2 d +1, \quad d_{A_n} = \frac{n(n+1)}{2} d - n(g-1),  
 \]
 \beq\la{ecz}
 \qquad d_{h_n} = \frac{n(n-1)}{2} d + n(g-1) +1, \qquad
d_{h_n} - d_{A_n} = -nd +2n(g-1) +1.
 \eeq
 
The corresponding formula for $SL_n$ follow easily, for example,  from the above, remembering that, in view of lemmata
\ref{lm1} and
\ref{lm2}, we have that
$\dim (M_n) = \dim (\check M_n) + h^0(D) + g$.
 \[
 d_{\check{M}_n} = n^2d-d, \quad d_{\check{A}_n} = \frac{n(n+1)}{2} d  -d - (n-1)(g-1), 
 \]
 \beq\la{ecce}
 d_{\check{h}_n} = \frac{n(n-1)}{2} d + (n-1)(g-1), \qquad 
 d_{\check{h}_n} - d_{\check{A}_n} = -(n-1)d +2(n-1) (g-1).
 \eeq
Recall that,
given $a \in A_n$,  we have been denoting the dimensions of $J_{n,a}, J_{n,a}^{\rm aff}$ and $J_{n,a}^{\rm ab}$
 by $d_a(J_n), 
d_a^{\rm aff}(J_n)$
and $d_a^{\rm ab}(J_n)$, respectively, and have been doing the same for    $\check{a}\in \check{A}_n \subseteq A_n$, $\check{J}_{n,a}$ etc.  (Recall that the Chevalley devissage is defined at geometric points, but the indicated dimensions
depend only on the underlying Zariski point; as it is about to become clear, it is better to keep track of Zariski points).

\subsection{Proof of  Theorem \ref{maintm}}\la{pmaintm}$\;$

Having done all the necessary preparation, the proof of Theorem \ref{maintm}  for the $SL_n$ socle, can now  proceed
parallel to the proof of Theorem \ref{ch-la-mt} for the $GL_n$ socle in \ci[\S9]{ch-la}.

 \n
 Let $a\in \check{A}_n$ belong to  ${{\socle}}({R  \check{h}_n }_*\oql)$. Apply the support inequality
 (\ref{99biz})  for the Zariski points in the socle:
 \beq\la{eq01}
  d_{\check{h}_n} - d_{\check{A}_n} + d_{a}  \geq d^{ab}_{a} (\check{J}_n).
 \eeq
 By  Lemma (\ref{azz}), we have $d^{ab}_{a} (\check{J}_n)=d^{ab}_{a} (J_n) -g$, so that
 \beq\la{eq02}
  d_{\check{h}_n} - d_{\check{A}_n} + d_{a}  \geq d^{ab}_{a} (J_n) -g.
  \eeq
  By combining (\ref{eq02}) with   (\ref{speremb}), we get
\beq\la{0112}
 d_{\check{h}_n} - d_{\check{A}_n} \geq \sum_{k=1}^s \left(d_{h_{n_k}} - d_{A_{n_k}} \right)
+ [d-(g-1)] -g.
\eeq

  \n
By using (\ref{ecce}) for  the l.h.s., and (\ref{ecz}) for each $n_k$ for the r.h.s., we  re-write (\ref{0112}) as follows
  \beq\la{t0}
  - (n-1)d +2(n-1)(g-1) \geq \sum_k \left[-n_k d + 2n_k(g-1)+1\right] + [d-(g-1)] -g.
  \eeq
  i.e. 
    \beq\la{t1}
  0 \geq \left(n-\sum_k n_k\right) \left[ d-2(g-1)\right] +(s-1).
  \eeq
  Since $d> 2(g-1)$ (by assumption) and $s\geq 1$ (by construction), we must have 
  $\sum_k m_k n_k =n=\sum_k n_k$ and $s=1.$ 
  
  \n
  The first condition forces  all $m_k=1$, so that the corresponding  geometric spectral curve $\m{C}_{\ov{a}}$  is reduced.

  \n
The second condition $s=1$ means that in addition to being reduced, the  geometric  spectral curve
  $\m{C}_{\ov a}$ must be  integral, i.e. $ a \in  \check{A}_n^{\rm ell}.$
\blacksquare

  We conclude  with two  remarks.

\begin{rmk}\la{poz} {\rm ({\bf Positive characteristic})}
Pierre-Henri Chaudouard has informed us that the main Theorem {\rm \ref{ch-la-mt}} for the $GL_n$ socle in \ci{ch-la}, should also hold over an algebraically
closed  field of positive characteristic bigger than $n$.  This should be the case in view of the fact that
one major obstacle in proving such theorem in positive characteristic had been the lack of the positive-characteristic analogue of the Severi inequality {\rm (\ref{98})}. At least as far as the corresponding inequality at the level of
the semiuniversal (miniversal) deformation for integral (even reduced) locally planar curves, this obstacle has been
removed in \ci[Thm 3.3]{me-ra-vi}. The restriction on the characteristic seems natural in view of the fact
that  the 
spectral covers have order $n$, and also because of formulae such as {\rm (\ref{sub})}. We did not verify whether
all of our arguments could be easily modified to yield the positive characteristic $(>n)$ cases of Theorems
{\rm \ref{ch-la-mt}} and  {\rm \ref{maintm}}, on the $GL_n$ and $SL_n$ socles.
\end{rmk}

\begin{rmk}\la{p} {\rm ({\bf $D=K_C$})}
The methods  of proofs of \ci{ch-la} for $GL_n$,  and of this paper for $SL_n$, do not work  in the very interesting case when $D=K_C$.
There is even more geometry at play in that symplectic/integrable case.  See \ci[\S11]{ch-la} for a short discussion
of the $D=K_C$ case.
\end{rmk}

\end{document}